\documentclass[a4paper,12pt]{article}

\usepackage{amssymb}
\usepackage{epsfig}
\usepackage{amsfonts}
\usepackage{amsmath}
\usepackage{euscript}
\usepackage{amscd}
\usepackage{amsthm}
\usepackage[all]{xy}
\DeclareMathAlphabet{\mathpzc}{OT1}{pzc}{m}{it}

\newcommand{\ra}{\rightarrow}

\newcommand{\lnd}{\mbox{\rm LND}}

\newtheorem{thm}{Theorem}[section]
\newtheorem{lem}[thm]{Lemma}
\newtheorem{prop}[thm]{Proposition} 
\newtheorem{cor}[thm]{Corollary}
\newtheorem{rem}[thm]{Remark}
\newtheorem{ex}[thm]{Example}

\newcommand{\m}{\mathpzc{m}}
\newcommand{\n}{\mathpzc{n}}
\newcommand{\p}{\mathpzc{p}}
\newcommand{\q}{\mathpzc{q}}

\newcommand{\BZ}{\mathbb Z}
\newcommand{\BQ}{\mathbb Q}

\newcommand{\BC}{\mathbb C}

\newcommand{\Ga}{\mathbb G}

\newcommand{\ot}{\mbox{$\otimes$}}

\title{Some $K$-theoretic properties of the kernel of  a locally nilpotent derivation on $k[X_1, \dots, X_4]$}
\author{S. M. Bhatwadekar{\footnote{{\small {\it 
Bhaskaracharya Pratishthana, 56/14 Erandavane, Damle Path,}}
{\small {\it Off Law College Road, Pune 411 004, India.}}
{\small {\it e-mail: smbhatwadekar@gmail.com}}}}, 
Neena Gupta{\footnote{{\small{\it Statistics and Mathematics  Unit, Indian Statistical Institute,}}
{\small{\it 203 B.T. Road, Kolkata 700 108, India.}}
{\small{\it e-mail: neenag@isical.ac.in}}}}
and Swapnil A. Lokhande{\footnote{{\small {\it 
Statistics and Mathematics  Unit, Indian Statistical Institute,}}
{\small{\it 203 B.T. Road, Kolkata 700 108, India.}}
{\small {\it e-mail: swaplokhande@gmail.com}}}}}

\begin{document}
\date{}
\maketitle
\begin{abstract}
Let $k$ be an algebraically closed field of characteristic zero,  $D$ a locally nilpotent derivation on
the polynomial ring $k[X_1, X_2,X_3,X_4]$ and $A$ the kernel of $D$.
A question of M. Miyanishi asks whether projective modules over $A$ are necessarily free.  
Implicit is a subquestion: whether the Grothendieck group $K_0(A)$ is trivial.
 
In this paper we shall demonstrate an explicit $k[X_1]$-linear fixed point free locally nilpotent
derivation $D$ of $k[X_1,X_2,X_3, X_4]$ whose kernel $A$ has an isolated singularity and
whose Grothendieck group $K_0(A)$ is not finitely generated; in particular, 
there exists an infinite family of pairwise non-isomorphic projective modules over the kernel $A$.
 
We shall also show that, although Miyanishi's original question does not have an affirmative
answer in general, suitably modified versions of the question do have affirmative answers
when $D$ annihilates a variable. For instance, we shall establish that in this case the groups $G_0(A)$ and $G_1(A)$
are indeed trivial.  Further, we shall see that if the above kernel $A$ is a regular ring, then
$A$ is actually a polynomial ring over $k$; in particular, by the Quillen-Suslin theorem, Miyanishi's
question has an affirmative answer.
 
Our construction involves rings defined by the relation $u^mv=F(z,t)$, where $F(Z,T)$ is an irreducible
polynomial in $k[Z,T]$. We shall show that a necessary and sufficient condition for such a ring
to be the kernel of a $k[X_1]$-linear locally nilpotent derivation $D$ of a polynomial ring $k[X_1,...,X_4]$
is that $F$ defines a polynomial curve.

\noindent
{\small {{\bf Keywords.} Locally Nilpotent Derivation, 
Polynomial Ring, Projective Module,\\ Grothendieck Group, Picard Group.}}
\end{abstract}

\section{Introduction}\label{intro}
Let $k$ be an algebraically closed field of characteristic zero and 
$R$ an integral domain containing $k$. 
A derivation $D:R \ra R$ is called {\bf locally nilpotent} if, for each $x\in R$, 
there exists an integer $s>0$ such that $D^s(x)=0$. The kernel of a derivation $D$
is the subring ker$(D)=\{b\in R\ |\ D(b)=0\}$ of $R$. 
For any subring $S$ of $R$, the set of all $S$-linear locally nilpotent derivations on $R$ 
is denoted by $\lnd_S(R)$. 
Note that, for any unit $\lambda \in R$, $D(\lambda)=0$ for any 
locally nilpotent derivation $D$ on $R$. In particular, 
$k \hookrightarrow$ ker$(D)$  and $D$ is $k$-linear for any locally nilpotent derivation $D$ on $R$.
A derivation $D$ is said to be {\bf fixed point free} if, the ideal 
generated by $DR$ is the unit ideal.
The concept of a locally nilpotent derivation and its kernel,
is an algebraic formulation of a $\Ga_a$-action on an 
affine variety and its ring of invariants.
Results on the kernels of locally nilpotent derivations 
have been crucial to solutions of certain central problems of affine algebraic geometry.

Now let $B=k[X_1,\ldots,X_n]$, a polynomial ring over $k$, $D$  a locally nilpotent
derivation of $B$ and $A={\rm ker}(D)$. As commented by M. Miyanishi
(\cite[Section 1.2]{M}),  $A$ inherits some of the properties (like factoriality)
of the polynomial ring $B$ and it is expected that $A$ enjoys, at least to some extent,
several other properties of $B$. Miyanishi listed a few properties of $B$ and asked
if $A$ satisfies them.  The following question of Miyanishi, asked in this spirit,
has also been highlighted by G. Freudenburg in his monograph on locally nilpotent derivations
(\cite[Section 11.17]{F}):
 
\smallskip

\noindent
{\bf Question}. { \it If $D\in \lnd_k(k[X_1,\ldots,X_n])$, then are all projective 
modules over the kernel of $D$ necessarily free?}
 
\smallskip
 
All the investigations in this paper have been directly or indirectly inspired by this question. 
Recall that, when $n\leq 3$,  it is known that the kernel of $D$ is a polynomial ring over $k$ (cf. \cite{R} and \cite{Mi})
and hence the above question has an affirmative answer by the Quillen-Suslin Theorem (cf. \cite{Q}, \cite{Su}).
However, when $n \ge 4$, the kernel of $D$ need not be a polynomial ring. 
In fact, when $n \ge 5$, there exist locally nilpotent derivations (in fact, triangular derivations)
for which the kernel is not even finitely generated (cf. \cite[Chapter 7]{F}).

The case $n=4$ is of special interest. In this case, it is not known whether 
ker$(D)$ is necessarily finitely generated in general. 
However, under the additional hypothesis that $D$ annihilates a variable of $k[X_1, X_2, X_3, X_4]$,
it is known that ker$(D)$ is indeed finitely generated (\cite{BD}).     
But even under this additional hypothesis, the  kernel of $D$ is not necessarily a
polynomial ring; in fact, D. Daigle and G. Freudenburg have shown (\cite{DF}) that, given an integer $r \ge 3$, there exists an element of 
$\lnd_{k[X_1]} (k[X_1, X_2, X_3, X_4])$ whose kernel cannot be generated by 
fewer than $r$ elements. 
 
In view of the above results, we investigate the question of Miyanishi (and closely related questions)
for the case $n=4$ and under the additional hypothesis that
$D$ annihilates a variable of $B=k[X_1,X_2,X_3,X_4]$, i.e., we investigate locally nilpotent derivations of $B$
``of rank less than or equal to 3'' (cf. \cite[p. 50]{F}). For convenience, we shall usually assume that $D(X_1)=0$, i.e.,
$D$ is $k[X_1]$-linear.

Our first main theorem addressing Miyanishi's question establishes that if the kernel $A$ of an element 
$D \in \lnd_{k[X_1]}(k[X_1, X_2,X_3,X_4])$ is a regular ring then $A$ is in fact  a polynomial ring. 
We prove (Theorem \ref{t5}):

\medskip

\noindent
{\bf Theorem I.} 
{\it Let $k$ be a field of characteristic zero, 
$D$ a locally nilpotent derivation of $k[X_1, X_2, X_3, X_4]$ such that $D(X_1)=0$ and
$A=$ker$(D)$. If $A$ is regular, then $A$ is a polynomial ring over $k[X_1]$. 
In particular, all projective modules over $A$ are free.
}

\smallskip

Theorem I gives a partial answer to a question of Freudenburg (Remark \ref{fr}(ii)).
We shall also describe (Theorem \ref{tr1}) the structure of the ring $A$  
(of Theorem I) when $A$ is not regular but
has only isolated singularities (i.e., $A_\m$ is regular for all but finitely many maximal ideals $\m$ of $A$). 
In contrast to the result of Daigle-Freudenburg (\cite{DF}) mentioned earlier, we shall see that in this case $A$ 
is generated by four elements.

Our next theorem in the context of Miyanishi's question highlights an interesting  $K$-theoretic property of 
$k[X_1, X_2, X_3, X_4]$ that is  shared by the kernel $A$ of any $k[X_1]$-linear locally nilpotent derivation 
of $k[X_1, X_2, X_3, X_4]$. We shall show that the groups $G_0(A)$ and $G_1(A)$ are indeed trivial. 
More precisely, we shall prove (Theorem \ref{th7}):

\medskip

\noindent
{\bf Theorem II.} 
{\it Let $k$ be an algebraically closed field of characteristic zero and 
$D$ be a $k[X_1]$-linear locally nilpotent derivation of $k[X_1, X_2, X_3, X_4]$. 
Let $A= $ker$(D)$. Then the canonical maps $G_j(k) \to G_j(A)$ are isomorphisms for $j= 0, 1$.
In particular, $G_0(A)= \BZ$ and $G_1(A) = k^*$. 
}
 
\smallskip
As a step, we shall first establish (Proposition \ref{propa}) certain conditions for a three-dimensional 
affine domain $A$ with generic fibre ${\mathbb A}^2$ to have trivial $G_0(A)$ and $G_1(A)$.

Finally, our following theorem (Theorem \ref{MT}) will show
that the answer to Miyanishi's question is not always affirmative in general.
 
\medskip

\noindent
{\bf Theorem III.} 
{\it Let $k$ be an algebraically closed field of characteristic zero.
Let $F(Z, T)$ be an irreducible polynomial of $k[Z, T]$, $C= k[Z, T]/(F)$ and 
$$
A= k[U, V, Z, T]/(U^mV- F(Z, T)), \text{~~where~~} m \ge 1.
$$
Then the following statements hold:
\begin{enumerate}
\item [\rm (i)] If $C$ is not a regular ring 
then $K_0(A)$ is not finitely generated; in particular,  there exists an infinite 
family of non-isomorphic projective $A$-modules of rank two
which are not even stably isomorphic.
\item [\rm (ii)] $A$ is isomorphic (as a $k$-algebra) to ker$(D)$ for some $D \in \lnd_{k[X_1]}(k[X_1, X_2, X_3, X_4])$ if and only if 
$C$ is a $k$-subalgebra of a polynomial ring $k[W]$. 
\end{enumerate}
}

In fact, we shall first show (Proposition \ref{pm}) that, 
under the notation of Theorem~III, the following statements are equivalent:
\begin{enumerate}
\item [\rm(a)] {\it $C$ is the coordinate ring of a non-singular rational curve}. 
\item [\rm(b)] {\it $K_0(A)= {\mathbb Z} $}.
\item [\rm(c)] {\it $K_0(A)$ is finitely generated.}
\end{enumerate}

Note that taking $C$ to be a $k$-subalgebra of a polynomial ring $k[W]$, 
a derivation $D$ obtained from (ii) of Theorem III can be easily extended 
to a locally nilpotent derivation of $k[X_1, \dots, X_n]$ ($n \ge 5$) 
having kernel $A[X_5, \dots, X_n]$. Therefore, if further $C$ is a non-regular ring, 
then these extensions of $D$ give counterexamples to Miyanishi's question 
for any $n \ge 4$ (cf. Remark \ref{r11}(c)). 
Example \ref{ex} gives an explicit fixed point free locally 
nilpotent derivation $D$ on $k[X_1, X_2, X_3, X_4]$ 
such that not all projective modules over ker$(D)$ are free.
 
We now give a layout of the paper. Section \ref{prelim} is on preliminaries
comprising mainly of concepts and results on $K$-theory.  
In Section \ref{regular}, we shall establish Theorem~I and Theorem \ref{tr1}. 
This section is independent of the $K$-theoretic 
results of Section \ref{prelim}. 
In Section \ref{g0g1},  we prove a few general results
including Proposition \ref{propa} and deduce Theorem II. 
In Section \ref{Main}, we shall first prove (Proposition \ref{pm})
that the Grothendieck group $K_0$ of a ring of the form
$A:= k[U, V, Z, T]/(U^mV-F(Z, T))$, 
where $m \ge 1$ and $F(Z, T)$ defines a singular polynomial curve, is not finitely generated. 
Next we shall construct a locally nilpotent derivation of $k[X_1, X_2, X_3, X_4]$
whose kernel is isomorphic to $A$ and thereby establish Theorem III.

\section{Preliminaries}\label{prelim}
Throughout the paper, we shall use the notation $R^{[n]}$ for a polynomial ring in 
$n$ variables over a commutative ring $R$. Thus, $E= R^{[n]}$ will mean that $E= R[t_1, \ldots, t_n]$, where 
$t_1, \ldots, t_n$ are elements of $E$ algebraically independent over $R$.
Unless otherwise stated, capital letters $X,Y,Z,T, U,V, W$ or $X_1, X_2, \dots, X_n$ will denote
indeterminates; thus a notation like $k[X_1, \dots, X_n]$ or $k[U,V,Z,T]$ will denote a polynomial ring
with the letters as variables. 

For any ring $R$, $R^*$ denotes the group of units of $R$.

Let $k$ be a field and $E$ an affine domain of dimension one over $k$. 
For brevity, we shall call the ring $E$ an {\it affine rational curve}
if its field of rational functions is isomorphic to $k(W)$ and 
a {\it polynomial curve} if $E\hookrightarrow k[W]$ 
for some $W$ transcendental over $k$. 
Note that if $E \hookrightarrow k^{[n]}$ then $E \hookrightarrow k^{[1]}$ (\cite[Lemma B]{E}). 

A smooth affine rational curve over an algebraically closed field $k$ is isomorphic
to $k[X, \frac{1}{f(X)}]$ for some polynomial $f(X) \in k[X]$ (\cite[Chapter 1, Exercise 6.1]{Ha}). 
A smooth polynomial curve (over any field $k$) is isomorphic to $k^{[1]}$ (cf. Lemma \ref{thE}). 

We now quote a few results which will be used in our proofs. 
We first state a theorem on ${\mathbb A}^2$-fibrations due to A. Sathaye (\cite[Theorem 1]{S}). 

\begin{thm}\label{Sa}
Let $(R,t)$ be a DVR containing ${\mathbb Q}$, with field of fractions $K$ 
and residue field $k$. Let $A$ be an integral domain containing $R$ such that 
$A[1/t]= K^{[2]}$ and $A/tA= k^{[2]}$. Then $A= R^{[2]}$. 
\end{thm}

The next theorem is a special case of the local-global theorem of Bass-Connell-Wright (\cite[Theorem 4.4]{BCW}).

\begin{thm}\label{bcw}
Let $R$ be a PID and $B$ a finitely generated $R$-algebra. Suppose that for all 
maximal ideals $\m$ of $R$, $B_\m= R_\m^{[n]}$ for some $n \ge 1$. 
Then $A =R^{[n]}$.
\end{thm}

The following theorem is due to Miyanishi (\cite{Mi}).

\begin{thm}\label{miya}
Let $k$ be a field of characteristic zero. 
Let $B = k[X,Y,Z]$, $0 \neq D\in \lnd_k(B)$ and  $A=$ ker$(D)$.
Then $A= k^{[2]}$. 
\end{thm} 

The following well-known result is taken from \cite[Theorem 1]{PE}.

\begin{thm}\label{thE}
Let $k$ be a field and $A$ a one-dimensional normal $k$-subalgebra of 
$k[X_1, \dots, X_n]$ for some $n \ge 1$. Then $A= k^{[1]}$. 
In particular, if $C$ is a polynomial curve then its integral closure
is $k^{[1]}$. 
\end{thm}

For convenience, we state below an easy lemma.  

\begin{lem}\label{lemlast}
Let $A\subseteq B$ be domains and $0\neq f\in A$. If $A_f=B_f$ and 
the induced map $A/fA\rightarrow B/fB$ is injective then $A=B$. 
\end{lem}

In the rest of this section we record a few definitions and state some 
well-known results from algebraic $K$-theory. 
Unexplained terms and other details can be found in  \cite[Chapter 4]{B2}. 

Let $\mathcal{R}$ be an Abelian category.  
Recall that a subcategory $\mathcal{E}$ of $\mathcal{R}$ is called {\it admissible} 
if it is a full additive subcategory with only a set of isomorphism classes of objects, and such that,
if $0\to M'\to M \to M'' \to 0$ is an exact sequence in $\mathcal{R}$ 
then $M, M'' \in \mathcal{E}$ implies $M' \in \mathcal{E}$.  

Let $R$ be a commutative Noetherian ring. Let $\mathcal{M}(R)$ denote 
the category of finitely generated $R$-modules. 
 Then $\mathcal{M}(R)$ is an Abelian category. 
Let $\mathcal{N}$ be an admissible subcategory of $\mathcal{M}(R)$. 

The {\it Grothendieck group} $K_0(\mathcal{N})$ of the category $\mathcal{N}$
is an Abelian group presented by generators $[M]$, 
where $[M]$ denotes the isomorphism class of objects of $\mathcal{N}$,
subject to the relations $[M]= [M']+ [M'']$ whenever there is an exact sequence
$0\to M'\to M\to M''\to 0$ in $\mathcal{N}$. 

Let $\mathcal{N}^{\BZ}$ denote the category 
whose objects are pairs $(M, \alpha)$, where $M$ is an object of 
$\mathcal{N}$ and $\alpha$ is an automorphism of $M$ in $\mathcal{N}$.  
A morphism $\phi: (M, \alpha) \to (M', \alpha')$ in $\mathcal{N}^{\BZ}$ is a morphism 
$\phi: M \to M'$ in $\mathcal{N}$  such that $\alpha'\phi = \phi\alpha$. 

The {\it Whitehead group} $K_1(\mathcal{N})$ of the category $\mathcal{N}$ is 
the quotient of the Grothendieck group of the category $\mathcal{N}^{\BZ}$
by the subgroup generated by  relations $[M, \alpha\beta]=  [M, \alpha]+ [M, \beta]$
whenever $M \in \mathcal{N}$ and $\alpha$ and $\beta$ are automorphisms of $M$.  
Hence, the group $K_1(\mathcal{N})$ is also an Abelian group. 

For the ring $R$, the groups $G_0(R)$ and $G_1(R)$ are, respectively, the Grothendieck group and the Whitehead group
of the category $\mathcal{M}(R)$. 

The category of finitely generated projective $R$-modules $\mathcal{P}(R)$ is an admissible 
subcategory of $\mathcal{M}(R)$. The groups $K_0(R)$ and $K_1(R)$ are, respectively, the 
Grothendieck group and the Whitehead group of the category $\mathcal{P}(R)$.
For any  field $k$, $G_0(k)= K_0(k)= \BZ$ and $G_1(k)=K_1(k) = k^*$.  

Note that if $\phi: A \to B$ is a ring homomorphism, then $\phi$ induces group homomorphisms 
$\tilde{\phi}: K_i(A) \to K_i(B)$ for $i=0, 1$. 
Let $\mathcal{P}_{<\infty}(R)$ denote the full subcategory of $\mathcal{M}(R)$ consisting of
$R$-modules with finite projective dimension.
From the definition of Grothendieck group it is easy to see that there is a  canonical 
group homomorphism
$$
K_0(R) (=K_0(\mathcal{P}(R))) \to K_0(\mathcal{P}_{<\infty}(R))
$$
which is an isomorphism by \cite[Corollary 8.5]{B2}.

Now let $u$ be a nonzerodivisor in $R$. Set

\medskip

\noindent
$\mathcal{P}_{<\infty}(R,u)$: the full subcategory of $\mathcal{P}_{<\infty}(R)$ consisting of $u$-torsion $R$-modules.  

\medskip

\noindent
$\mathcal{P}_{1}(R,u)$: the full subcategory of $\mathcal{P}_{<\infty}(R)$ consisting of $u$-torsion $R$-modules $M$
with projective dimension $Pd_R(M)=1$. 

\begin{rem}\label{r3}
{\em As before, there are canonical group homomorphisms 
$$
K_0(\mathcal{P}_{<\infty}(R,u)) \to K_0(\mathcal{P}_{<\infty}(R)) \text{~~and~~}
K_0(\mathcal{P}_{1}(R,u)) \to K_0(\mathcal{P}_{<\infty}(R,u)).
$$
Moreover the homomorphism $K_0(\mathcal{P}_{1}(R,u)) \to K_0(\mathcal{P}_{<\infty}(R,u))$ 
is an isomorphism by \cite[Lemma 10.3(b)]{B2}. 
Thus, we have a canonical map $\delta: K_0(\mathcal{P}_{1}(R,u)) \to K_0(R)$ such that
if $M$ is a $u$-torsion $R$-module with projective dimension $Pd_R(M)=1$, then 
$\delta([M])=[P_0]-[P_1]$ where $0\rightarrow P_1\rightarrow P_0\rightarrow M\rightarrow 0$ 
is a projective resolution of $M$.}
\end{rem}

The exactness of the following localization sequence is proved in \cite[Theorem 10.1]{B2}.

\begin{thm}\label{locseq}
Let the notation be as above and $S= \{u^n~|~n \ge 0\}$. There is a unique homomorphism 
$$
\partial: K_1(R[1/u])\rightarrow K_0(\mathcal{P}_{1}(R,u))
$$
such that $\partial([S^{-1}P, \alpha])=[P/\alpha P]$ whenever $P \in \mathcal{P}(R)$, 
$\alpha \in GL(S^{-1}P)$, and $\alpha(P)\subset P$. Moreover, the 
following sequence is exact. 
\begin{equation}\label{eqn}
\xymatrix 
{ K_1(R) \ar@{->}[r] &  K_1(R[1/u]) \ar@{->}[r]^-{\partial} & K_0(\mathcal{P}_{1}(R,u))\ar@{->}
[r]^-{\delta} & K_0(R)\ar@{->}[r] & K_0(R[1/u])}.
\end{equation}
\end{thm}
                                                  
Any finitely generated projective $R/uR$-module $P$ 
can be regarded as a $u$-torsion $R$-module with $Pd_R(P) =1$.
Hence there exists a well-defined map 
\begin{equation}\label{eqka}
K_0(u): K_0(R/uR)\to K_0(\mathcal{P}_{1}(R,u)) \text{~~such that~~}  K_0({u})([Q])=[Q]
\end{equation}                                                             
where $Q$ is a finitely generated projective $R/uR$-module.
When $R$ is a polynomial ring in one variable over its subring $C$ 
with $u$ as a variable, we have the following result 
(cf. \cite[Proposition 11.3]{B2}). 

\begin{lem}\label{l1}
Suppose that $R= C[u] =C^{[1]}$ is a polynomial ring over $C$. Then
the homomorphism $K_0(u): K_0(R/uR)\to K_0(\mathcal{P}_{1}(R, u))$ 
defined in equation (\ref{eqka}) is injective. 
\end{lem}

\begin{rem}
{\em The map in equation (\ref{eqka}) is not always injective. 
For instance, if $A= \BC[X, Y, Z, T]/(X^2Y+X+Z^2+T^3)$ and 
$x$ denotes the image of $X$ in $A$, then the map 
$K_0(x): K_0(A/xA)\to K_0(\mathcal{P}_{1}(A,x))$ is not injective.

For, $K_0(A/xA)= K_0(C^{[1]})$, where $C=\BC[Z,T]/(Z^2+T^3)$. 
Since $C$ is non-regular, ${\rm Pic}(C)$ is not finitely generated (cf. Lemma \ref{lpic1}) and 
hence $K_0(A/xA)$ is not a finitely generated group. We now show that 
$K_0(\mathcal{P}_{1}(A,x)) \cong \BZ$.  This will prove that the map 
$K_0(x)$ cannot be injective. 

Since $A$ is a regular ring, any finitely generated 
$A/xA$-module has finite projective dimension when treated as an $A$-module. 
Therefore $\mathcal {M}(A/xA)$ can be regarded as a full subcategory of
$\mathcal{P}_{< \infty}(A,x)$ such that the canonical map
$K_0(\mathcal{M}(A/xA)) \rightarrow K_0(\mathcal{P}_{< \infty}(A,x))$ defined by sending  
$[M]$ to $[M]$ is an isomorphism. 
By definition, $G_0(A/xA)=K_0(\mathcal{M}(A/xA))$; 
by Remark \ref{ktheory}(viii), $G_0(A/xA) (= G_0(C^{[1]}))= G_0(C)$; 
and, by  Lemma \ref{lemg}(i), $G_0(C) \cong \BZ$.
Hence, since $K_0(\mathcal{P}_{1}(A,x)) \cong K_0(\mathcal{P}_{< \infty}(A,x))$ by Remark \ref{r3},
we have $K_0(\mathcal{P}_{1}(A,x))\cong \BZ$.
}
\end{rem}

The following lemma occurs in \cite[Theorem 1.1]{We}. 

\begin{lem}\label{ll1}
Let $\phi:A\to B $ be a ring homomorphism. Let $u\in A$ and $\phi(u)\in B$ be nonzerodivisors of $A$ and $B$ respectively. 
Then the map $K_0(\tilde{\phi}): K_0(\mathcal{P}_{1}(A,u))\to K_0(\mathcal{P}_{1}(B,\phi(u)))$ sending 
$[M]$ to $[M\ot_AB]$ is well defined. 
\end{lem}

The following result on the Picard group of a curve  (\cite[p. 258]{E}, \cite[Theorem 3.2]{Wi}) will be 
used in our proofs of Lemma \ref{polycur} and Theorem \ref{pm}.  
Recall that for an integral domain $R$ with field of fractions $F$, 
$\widetilde{K}_0(R)$ is the kernal of the group homomorphism $K_0(R)\to \mathbb{Z}$ defined by 
$[P]\mapsto \mbox{dim}_F(P\ot_R F)$. Moreover, if $C$ is a one-dimensional domain, then 
${\rm Pic}(C) \cong \widetilde{K}_0(C)$. 

\begin{lem}\label{lpic1}
Let $k$ be an algebraically closed field of characteristic zero and $C$ be the coordinate ring of an irreducible affine curve over $k$. 
Then ${\rm Pic}(C)$ is finitely generated if and only if $C$ is a non-singular affine rational curve.
As a consequence, $K_0(C)$ is finitely generated if and only if $C$ is a non-singular affine rational curve.
\end{lem}

We record below some well-known facts on the groups $G_0(R)$ and $G_1(R)$. 
For the definition of $G_i(R)$ for $i\ge 2$, we refer to \cite[Chapters 4 and 5]{Sr}. 
On first reading, a reader may skip this and refer to it whenever necessary. 
\begin{rem}\label{ktheory}
{\em Let $R$ be a Noetherian ring. 

(i) If $R$ is an integral domain with field of fractions $K$,  
then there exists a group homomorphism $G_0(R)\to \mathbb{Z}$ defined by 
$[M]\mapsto \mbox{dim}_K(M\ot_R K)$. It is called the rank map. 

(ii) The group $G_0(R)$ is generated by $[R/\p]$, 
where $\p$ varies over ${\rm Spec}~R$. For, any finitely generated module $M$ 
over the Noetherian ring $R$ has a filtration of
submodules $M=M_0\supset M_1\supset M_2\supset \ldots \supset M_r=0$
satisfying $M_i/M_{i+1}\cong R/\p_i$ for some prime ideals $\p_i$ of $R$, so that 
$[M]=\sum_i [R/p_i]$ (cf. \cite[Proposition 4.4]{B2}).

Let $M$ be a finitely generated $R$-module with filtration
$M=M_0\supset M_1\supset M_2\supset \ldots \supset M_n=0$ and $\alpha \in {\rm Aut}_R(M)$
with $\alpha(M_i) =M_{i}$ for $0\le i \le n$. Let $N_i= M_i/M_{i+1}$ and $\alpha_i$ 
be the induced automorphism of $N_i$ for $0\le i < n$. Then $[M, \alpha] = \sum_{i=0}^{n-1} [N_i, \alpha_i]$ (cf. \cite[Proposition 4.6]{B2}). 

(iii) If $\p$ is a principal ideal of $R$ generated by a nonzerodivisor $p$ of $R$ then $[R/\p]=0$ in $G_0(R)$
as there exists a short exact sequence $0 \to R \stackrel{p}{\rightarrow} R \to R/\p \to 0$ and 
$[R/\p]=[R]-[R]$ in $G_0(R)$.

(iv) There is a canonical group homomorphism $\theta: R^* \to G_1(R)$ defined by 
$\theta(u) = [R, \tilde{u}]$, where $\tilde{u}: R \to R$ is the $R$-linear 
automorphism defined by $\tilde{u}(r)=ru$ $\forall$ $r \in R$.
In particular, if $R$ is a finitely generated algebra over a field $k$, 
then there is a canonical map $\theta: k^* \to G_1(R)$.  

(v) If $\imath: R \to S$ is a ring homomorphism such that $S$ is a finite $R$-module, 
then $\imath$ induces, via restriction of scalars, group homomorphisms ${\imath}_*: G_i(S) \to G_i(R)$ for each $i\ge 0$.

(vi) Let $S$ be a Noetherian ring and $\imath: R \to S$ be a flat ring homomorphism.  
Then $\imath$ induces group homomorphisms ${\imath}^*: G_i(R) \to G_i(S)$ for $i\ge 0$ 
induced by the functor $\otimes_R S: \mathcal{M}(R) \to \mathcal{M}(S)$ defined by
$M \to M\otimes_R S$ (cf. \cite[5.6 p. 52 and 5.8 p. 53]{Sr}). 
In particular, if $k$ is a field and $R$ is a finitely generated $k$-algebra, and 
$j: k \hookrightarrow R$ is the canonical inclusion map,  
then there exists a canonical group homomorphism ${j}^*: G_i(k) \to G_i(R)$ for $i\ge 0$
and, for $i=1$, this is the map $\theta: k^* \to G_1(R)$ defined in (iv). 

(vii) Let $x$ be a nonzerodivisor of $R$, $j: R \to R[x^{-1}]$ the inclusion map, and 
$\pi: R \to R/xR$  the canonical map.
Then we have the following long exact sequence of groups:
\begin{equation*}
{\longrightarrow} G_{i}(R/xR) \stackrel{\pi_*}{\longrightarrow} G_i(R) \stackrel{j^*}{\longrightarrow} G_i(R[1/x])
\stackrel{\delta}{\longrightarrow} G_{i-1}(R/xR) \to  \dots\to G_0(R[1/x]) \to 0.
\end{equation*}
Moreover, if $\phi: R\to S$ is a flat ring homomorphism with $u= \phi(x)$, 
then we have the natural commutative diagram:
\[
\xymatrix{
\dots \ar[r] & G_i(R/xR) \ar[r] \ar[d] & G_i(R) \ar[r] \ar[d] & G_i(R[1/x]) \ar[r]^{\delta} \ar[d]  & G_{i-1}(R/xR) \ar[r]\ar[d]& \dots   \\
\dots \ar[r] & G_i(S/u S) \ar[r] &  G_i(S) \ar[r] & G_i(S[1/u]) \ar[r]^-{\delta}  & G_{i-1}(S/uS) \ar[r] & \dots , 
}
\]
where the vertical maps are induced by $\phi$ (cf. \cite[Proposition 5.15, 5.6 p.52 and 5.16 p. 61]{Sr}). 

(viii) For any indeterminate $T$ over $R$, the maps 
$$
G_i(R) \stackrel{}{\longrightarrow} G_i(R[T])
$$
are isomorphisms for all $i \ge 0$. Let $j: R \rightarrow R[T,T^{-1}]$ denote the inclusion map.
Then the induced maps
$$
\jmath^*: G_i(R) \to G_i(R[T,T^{-1}])
$$ 
are split inclusions for all $i \ge 0$. For $i=0$, $\jmath^*$ is an isomorphism and, 
for $i \ge 1$, $G_i(R[T,T^{-1}]) \cong G_i(R) \oplus G_{i-1}(R)$ (cf. \cite[Theorem 5.2]{Sr}).     

(ix) If $R$ is a regular ring, then $G_i(R) = K_i(R)$ $\forall$ $i \ge 0$ (cf. \cite[Theorem 4.6]{Sr}).
In particular, $G_1(k[T]) = K_1(k[T])= k^*$  and $G_1(k[T, f(T)^{-1}])= K_1(k[T, f(T)^{-1}])=k[T, f(T)^{-1}]^*$
for any field $k$, $T$ transcendental over $k$ and $f(T) \in k[T]$. 
}
\end{rem}

\section{Theorem I}\label{regular}
In this section, we deal with a situation where Miyanishi's question
has an affirmative answer; in fact we prove something stronger (Theorem \ref{t5}).
All the derivations considered in this section are non-zero. 

Let $(R,t)$ be a DVR containing $\mathbb Q$ with field of fractions $K$ 
and residue field $k$ and let $A$ be the kernel of a non-zero 
$R$-linear locally nilpotent derivation $D$ of $R[X,Y,Z]$.  
Note that, since $A$ is  factorially closed in $R[X,Y,Z]$ (cf. \cite[pg. 22]{F}),  
we have $A$ is a UFD, $t$ is a prime element of $A$ and  $A/tA$ is a
$k$-subalgebra of $k[X,Y,Z]$. 
Moreover, by Theorem \ref{miya}, $A \otimes_R K = K^{[2]}$. 
Therefore there exist $ f_0, f_1 \in A$ such that $ R[1/t][f_0,f_1] =
 A[1/t] = K^{[2]}$. It is easy to see that  we can choose $ f_0 \in A$ such that its image
$\overline {f_0}$ in $A/tA$ is transcendental over $k$. 
Keeping these facts in mind we state a consequence of 
\cite[Proposition 4.13]{BD} which presents a description of the kernel $A$.

\begin{lem}\label{BD}
Let $(R,t)$ be a DVR containing $\mathbb Q$ with field of fractions $K$ 
and residue field $k$. Let $B = R [X,Y,Z]$, $D (\neq 0) \in \lnd_R(B)$ and  $A=$ ker$(D)$.
Let $ f_0, f_1 \in A$ be such that $A[1/t]= R[1/t][f_0, f_1] =K^{[2]}$.
Moreover assume that $R[f_0, f_1] \neq A $ and that 
the image $\overline {f_0}$ of $f_0$ is transcendental over $k$. 
Then there exists a generating set $\{f_0, f_1, \cdots, f_{n+1}\}$ of $A$ over $R$
such that, setting $A_i:= R[f_0,f_1, \dots, f_{i}]$, 
for $0\le i\le n+1$,  we have, for each $i$,  $1\le i\le n$,
\begin{enumerate}
\item [\rm (i)]  $Q_i: = tA \cap A_i= (t, h_i)A_i$ for some $h_i \in A_i$. Further $Q_i$ is a height two prime ideal of $A_i$. 
\item [\rm (ii)] $f_{i+1} = h_i/t \in A$.
\item [\rm (iii)] $A_{i+1} \cong_{A_{i}} A_{i}[V_{i+1}]/(tV_{i+1}-h_{i})$,  where $A_i[V_{i+1}]= {A_i}^{[1]}$.
\item [\rm (iv)]  $A_i$ is a UFD and $t$ is a prime element of $A_i$.
\end{enumerate}
In particular, $E:= A_n$ is a subring of $A$ finitely generated over $R$ such that 
$E[1/t]= A[1/t]= K^{[2]}$, $E/tE$ is a domain and $A\cong_{E} E[V]/(tV-g)$ where $g= h_n$. 
As a consequence, $A/tA = C^{[1]}$ where $C = E/(t,g) (\hookrightarrow k[X,Y,Z])$ is a polynomial curve over $k$.
\end{lem}

\begin{proof}
(i) and (ii) follow from \cite[Lemma 4.12 and Proposition 4.13]{BD}.

We now show (iii) and (iv) by induction on $i$. 
Since $A_1= R[f_0, f_1] = R^{[2]}$, $A_1$ is a UFD and $t$ is a prime element in $A_1$.
By induction hypothesis, we assume that $A_i$ is a UFD and $t$ is a prime element of $A_i$. 
Since $tA \cap A_i = Q_i= (t, h_i)A_i$ is a height two prime ideal of $A_i$, 
we have  $h_i \notin tA_i$. Hence, $tV_{i+1}-h_i$ is a prime element in 
$A_i[V_{i+1}](= {A_i}^{[1]})$.  As $A_{i+1}= A_i[f_{i+1}]= A_i[h_i/t]$, it follows that 
$tV_{i+1}-h_i$ generates the kernel of the surjective $A_{i}$-algebra homomorphism from 
$A_i[V_{i+1}]\to A_{i+1}$ sending $V_{i+1}$ to $f_{i+1}$. 
Thus, $A_{i+1}\cong_{A_i} A_i[V_{i+1}]/(tV_{i+1}-h_i)$.
Hence, since $A_{i+1}/(t)\cong_{A_i/tA_i} A_{i}[V_{i+1}]/(t, h_i)$ is an integral domain, 
we have $t$ is a prime element of $A_{i+1}$. 
Since $A_{i+1}[1/t]= A_1[1/t]= R[1/t]^{[2]}$, a UFD,  
it follows by Nagata's criterion that $A_{i+1}$ is a UFD (cf. \cite[Theorem 20.2]{Mat}). 

Set $E:= A_n$. Since $A=A_{n+1}$ and $n \ge 1$, we have $R[f_0, f_1] \subseteq E$ and hence $E[1/t]= A[1/t]= K^{[2]}$. 
By (iii) and (iv), $E/tE$ is a domain and $A\cong_{E} E[V]/(tV-g)$ where $g= h_n$ and $E[V]= E^{[1]}$.
Further since $C= E/(t,g) \hookrightarrow A/tA \hookrightarrow k[X,Y, Z]$, we have $C$ is a polynomial curve. 
\end{proof}

\begin{lem}\label{b1}
Let the notation and hypothesis be as in Lemma \ref{BD}. 
Let $\n$ be a  maximal ideal of $E$ containing $(t,g)$. Then
\begin{enumerate}
\item [\rm (i)] $t \not\in \n^2$.
\item [\rm(ii)] If $E_{\n}$ is not regular then there exist infinitely many maximal ideals $\m$ of $A$ containing $t$
for which $A_\m$ is not regular. 
\end{enumerate}
\end{lem}
\begin{proof}
(i) Suppose, if possible, that $t \in \n^2$. Set $Q:= (t, g)E$. 
By Lemma \ref{BD}, $C(=E/Q) \hookrightarrow A/tA \hookrightarrow B/tB= k[X,Y,Z]$. 
Hence, by Theorem \ref{thE}, the normalisation of $C$ is $k[W]$ for some $W$ transcendental over $k$.    
Thus $E/Q = C \hookrightarrow k[W] \hookrightarrow  B/tB = k[X,Y,Z]$ and  
there exists a prime element $p \in k[W]$ such that $pk[W]$ lies over the prime ideal $\n/Q$ of $E/Q$.
As $(B/tB)^* = k^*= (k[W])^*$, $pB/tB$ is contained in some height one prime ideal of $B/tB$. 
Therefore, there exists a  prime ideal $P$ of $B$ of height two containing $t$ such that
$\n/Q = E/Q \cap P/tB$ and hence $\n = E \cap P$.
Therefore, as $t \in \n^2$, we have $t \in P^2$. 
But this is absurd as $B = R[X,Y,Z]$. 
Hence $t \notin \n^2$.

(ii) Suppose that $E_{\n}$ is not a regular ring. Then 
$\dim_k \n/\n^2\ge 4$. 
Fix $\lambda \in \BQ$. Then $M_{\lambda}:=(\n, V-\lambda)$ is a maximal ideal of $E[V]$  
such that $(t, g) \subseteq M_\lambda$. 
Hence, $M_\lambda$ induces a maximal ideal $\m_\lambda$ of $A (\cong E[V]/(tV-g))$.
Since $V-\lambda \notin M_{\lambda}^2$,   $\dim_k M_{\lambda}/M_{\lambda}^2 \ge 5$, and hence 
$\dim_k \m_\lambda/\m_\lambda^2 \ge 4$, i.e.,  $A_{\m_\lambda}$ is not a regular ring.     
Clearly, if $\lambda \neq \lambda'$ then ${\m_\lambda} \neq {\m_\lambda'}$. Hence the result.
\end{proof}

The following result proves a local version of Theorem I.

\begin{prop}\label{bdd2}
Let $(R,t)$ be a DVR containing $\mathbb Q$ with field of fractions $K$ 
and residue field $k$. Let $B = R [X,Y,Z]$, $D (\neq 0) \in \lnd_R(B)$ and  $A=$ ker$(D)$.
\begin{enumerate}
\item[\rm(i)] If $A$ has at most isolated singularities, then there exists a subring $E$ of $A$ and an element $g \in E$ such that  
$E = R^{[2]}$, $E[1/t]= A[1/t]$, $tA \cap E= (t, g)E$ and $A = E[g/t]$.
\item[\rm (ii)] If $A$ is a regular ring, then $A = R^{[2]}$.
\end{enumerate}
\end{prop}
\begin{proof} 
Let $f_0, f_1$ be as in Lemma \ref{BD}. If $A= R[f_0, f_1]$, 
then taking $E= R[f_0, f_1]$ and $g=0$, 
we will be through. So we suppose that $A\neq R[f_0, f_1]$. 
Let $E$, $g$ and $C$ be as in Lemma \ref{BD}. Then 
$E[1/t]= A[1/t]$, $tA \cap E= (t, g)E$ and $A = E[g/t]$.  

(i) We show that $E = R^{[2]}$.  If $E = R[f_0,f_1]= R^{[2]}$ then there is nothing to prove. 
So we assume that $R[f_0,f_1] \not=E$, i.e., $E = A_n$ with $n \geq 2$. 
Since $E[1/t] = {R[1/t]}^{[2]}$, by Theorem \ref{Sa}, 
it is enough to show that $E/tE= k^{[2]}$. 
By Lemma \ref{BD},  $tA \cap A_{n-1}=(t,h_{n-1})A_{n-1}$ is a prime ideal of $A_{n-1}$ and 
$E (= A_n) \cong A_{n-1}[V_n]/ (t{V_n} - h_{n-1})$.   
Set $C': = A_{n-1}/(t,h_{n-1})$. Then $E/tE \cong C'[V_n]$. 
Hence it is enough to show that $C'= k^{[1]}$. 

Note that $C= E/(t, g)= C'[V_n]/(\overline{g})$, where $\overline g$ denotes the image of $g$ in $E/tE$. 
Further, since $C=E/(t, g)$ is a polynomial curve and 
$(t,h_{n-1})A_{n-1} = tA \cap A_{n-1} = tA \cap E \cap A_{n-1}= (t, g)E \cap A_{n-1}$, we have
$$
k \subsetneqq C' \hookrightarrow C  \hookrightarrow k[W]= k^{[1]}
$$
for some $W$. Since $C'$ contains a non-constant polynomial in $W$, we have $W$ is integral over $C'$. 
In particular,  $C$ is integral over ${C'}$ and
$\overline g$ is a monic polynomial in $V_n$ with coefficients in $C'$.
Hence, given a maximal ideal $\n'$  of $C'$, there exists a maximal
ideal ${\overline  \n}$ of $E/tE (= C'^{[1]})$ containing $\overline{g}$
such that $\n' = C'\cap {\overline \n}$
so that $({E/t E})_{\overline \n}$ is a faithfully flat extension of
${C'}_{\n'}$.  Let $\n$ be the lift of $\overline{\n}$ in $E$. Then $(t, g)E \subseteq \n$. 
Since $A_\m$ is regular for all but (at most) finitely many maximal ideals $\m$ of $A$,
by Lemma \ref{b1}, we have $t \notin \n^2$ and $E_\n$ is a regular ring. 
Hence, $({E/tE})_{\overline \n}$ is a regular ring.
Therefore $C'_{\n'}$ is a regular ring (cf. \cite[Theorem 23.7(i)]{Mat}). This being true for every maximal ideal $\n'$ of $C'$,
we have $C'$ is a regular ring and hence, by Theorem \ref{thE}, $C'= k^{[1]}$. 
Hence $E= R^{[2]}$. 

(ii) We now assume further that $A$ is regular and show that $A=R^{[2]}$.
By Theorem \ref{miya},  $A[1/t] = K^{[2]}$. Hence, by 
Theorem \ref{Sa}, it is enough to prove that $A/tA = k^{[2]}$. 
Since $A/tA = C^{[1]}$ and $C$ is a polynomial curve over $k$, by Theorem \ref{thE}, 
it is enough to prove that $C$ is a regular ring.

By (i), $E= R^{[2]}$, say $E= R[X_1, X_2]$.
Thus $C = k[X_1, X_2]/(\overline g)$, where $\overline g $ is the image of $g$ in $E/tE$. 
Suppose, if possible, that  $C$ is not regular. 
Then there exists a maximal ideal ${\tilde \n}$ of $k[X_1, X_2]$ such that ${\overline g} \in {\tilde \n}^2$. 
Let $\n$ be the lift of ${\tilde \n} $ in $E (= R[X_1, X_2])$. Then $g \in tE + \n^2$.
Let $ e \in E , f \in {\n}^2  $  be such that $ g = te + f$. Therefore the
element $tV - g =t(V - e) + te -g = t(V-e)- f \in M^2$ where $M = (\n, V-e) $ is a
maximal ideal of $E[V] = R[X_1, X_2, V]$. 
Identifying A with $E[V]/(tV-g)$, we see that $\m=M/(tV-g)$
is a maximal ideal of $A$ for which $A_{\m}$ is not regular.
This is  a contradiction. Hence $C$ is
regular and therefore $A = R^{[2]}$.
\end{proof}

\begin{rem}
{\em Let the notation and hypothesis be as in Lemma \ref{BD}. Set $C_i:= A_i/Q_i$ for $1\le i\le n$. 
Following the arguments as in the proof Proposition \ref{bdd2}(i) it can be shown that 
$
C_1 \hookrightarrow C_2 \hookrightarrow \dots \hookrightarrow C_n \hookrightarrow k^{[1]}
$
and that $C_i$ is faithfully flat over $C_{i-1}$ for each $i$, $2 \le i \le n$.  
Now suppose that $A= A_{n+1} = R^{[2]}$.
Then since $A/tA(= k^{[2]})= {C_n}^{[1]}$, we have 
$C_n$ is a regular ring. Hence, $C_i$ is regular for each $i$ (cf. \cite[Theorem 23.7(i)]{Mat})
and therefore, by Theorem \ref{thE}, $C_i= k^{[1]}$ $\forall~i$, $1\le i\le n-1$.  
Hence $A_i/tA_i= {C_{i-1}}^{[1]}= k^{[2]}$ and since $A_i[1/t]= K^{[2]}$, by  
Theorem \ref{Sa} we have $A_i= R^{[2]}$ for each $i$, $2\le i \le n$.
}
\end{rem}

We now prove Theorem I. 

\begin{thm}\label{t5}
Let $k$ be a field of characteristic zero, $D(\neq 0)\in \lnd_{k[T]}(k[T, X, Y, Z])$ and $A=$ker$(D)$.
Suppose that $A$ is regular. 
Then $A= k[T]^{[2]}$.  In particular,  
all finitely generated projective modules over $A$ are free. 
\end{thm}
\begin{proof}
Set $R:= k[T]$ and $B: = k[T, X, Y, Z]$. Let $P$ be a maximal ideal of $R$ and $S=R \setminus P$. 
Set $R_P:= S^{-1}R$, $B_P:= S^{-1}B$ and $A_P:= S^{-1}A$.   Then $R_P$ is  a DVR.  
Since $D(T)= 0$,  $D$ induces an $R_P$-linear locally nilpotent derivation $D_P$
of $B_P$  with kernel $A_P$. Hence, by Proposition \ref{bdd2}(ii), $A_P= {R_P}^{[2]}$. 
Therefore, since $R$ is a PID, by Theorem \ref{bcw}, $A=  R^{[2]}= k[T]^{[2]}$. 
Thus, by the Quillen-Suslin Theorem (cf. \cite{Q},\cite{Su}), 
all finitely generated projective $A$-modules are free. 
\end{proof}

\begin{rem}\label{fr}
{\em
(i) Theorem \ref{t5} shows that if 
$D$ is a triangular derivation of $k^{[4]}$ and 
ker$(D)$ is a regular ring then ker$(D)$ is a polynomial ring. 
This result does not extend to $k^{[5]}$. 

Consider Winkelmann's derivation (\cite[3.9.5, p. 75]{F}) 
$D: k[X,Y,U,V,Z] \to  k[X,Y,U,V,Z]$ defined by 
$$
D(X)=D(Y)=0, \,  D(U)= Y, \,  D(V)= X \text{~and~} D(Z)= 1+ XU-YV.
$$ 
Then $A:= {\rm ker}(D)= k[X,Y, f, g, h]$, where 
$f = XU-YV$, $g= YZ-(1+f)U$ and $h= XZ-(1+f)V$. 
It is easy to see that $A \cong_k k[X,Y,F,G,H]/(YH-XG-(1+F)F)$
which is a regular ring but not a 
polynomial ring. In fact, N. Mohan Kumar and M. Nori proved that
$\widetilde{ K}_0(A)  = {\BZ}$ (cf. \cite[Lemma 17.2 and Corollary 17.3]{Sw}), 
showing that there exists a projective
$A$-module which is not even stably free.

(ii) Consider the Russell-Koras cubic $A=\BC[X, Y, Z, T]/(X^2Y+X+Z^2+T^3)$.
L.~Makar-Limanov has shown (cf. \cite[Theorem 9.6]{F}) that $A \neq \BC^{[3]}$ and 
S. Kaliman has asked (\cite[Section 11.11]{F}) whether $A^{[1]}= \BC^{[4]}$.  
Now, for the affine $3$-space ${\mathbb A}^3_k$, while the Zariski Cancellation Problem
has a negative solution when ${\rm ch}~k >0$ (\cite{G}), 
it is still open when ${\rm ch}~k=0$. 
If $A^{[1]}= \BC^{[4]}$, then the Russell-Koras cubic $A$ will give a 
negative solution to the Zariski Cancellation Problem even in  characteristic zero.  

In this connection G. Freudenburg posed a stronger question (\cite[Section 11.11]{F}):
whether the Russell-Koras cubic $A$ is the kernel 
of a locally nilpotent derivation of $\BC^{[4]}$. 
Theorem \ref{t5} shows that the Russell-Koras cubic $A$, being regular but not $\BC^{[3]}$,
cannot be isomorphic to the kernel of any locally nilpotent 
derivation $D$ of $\BC[X_1, X_2,X_3,X_4]$ which annihilates a variable of $\BC[X_1, X_2,X_3,X_4]$.
}
\end{rem}

We shall now describe a structure of the kernel of a  $k[X_1]$-linear locally nilpotent derivation of $k[X_1, \dots, X_4]$
when it has only isolated singularities. We first state a lemma which is a semilocal version of Proposition \ref{bdd2}(i).

\begin{lem}\label{bl3}
Let $R$ be  be a semilocal PID containing ${\BQ}$.
Let $\{t_1 R, \dots, t_nR \}$ be the set of all maximal ideals of $R$ and $t = t_1 \cdots t_n$. 
Let $D(\neq 0) \in \lnd_R(R[U,V,W])$ and $A=$ ker$(D)$.   
Suppose that $A$  has only isolated singularities. 
Then there exists a subring $E$ of $A$ and an element $f \in E$ such that $E = R^{[2]}$, 
$A[1/t]=E[1/t]$, $t A \cap E = (t,f)E$ and $A =  E[f/t]$.
\end{lem}
\begin{proof}
We prove the result by induction on $n$. 
If $n =  1$, then we are through by Proposition \ref{bdd2}(i).
We assume that $n \geq 2 $. 

Let $t' = t_2\cdots t_n$, $ R_1 = R[1/t_1]$ and $A_1 =  A[1/t_1]$. 
Then $R_1$ is a semilocal PID having $n-1$ maximal ideals. 
Hence, by induction hypothesis, there exists a subring 
$E_1$ of  $A_1$ and an element $f_1 \in E_1$ 
such that $E_1 = R_1^{[2]}$, $E_1[1/t']=A_1[1/t']$, 
${t'}A_1 \cap E_1 =({t'}, f_1)E_1$ and $A_1 = E_1[f_1/{t'}]$.

Let $R_2 = R[1/{t'}]$ and $A_2 = A[1/{t'}]$. Then $R_2$ is
DVR and hence, by Proposition \ref{bdd2}(i), 
there exists a subring $E_2$ of $A_2$ and an element $f_2 \in E_2$ such that 
$E_2 = R_2^{[2]}$, $E_2[1/t_1]= A_2[1/t_1]$, $t_1A_2 \cap E_2 = (t_1, f_2)E_2$ and
$A_2 = E_2[f_2/t_1]$.

Let $E = E_1 \cap E_2$. Since $(t_1, t')R=R$, we have $A = A_1 \cap A_2$. 
Hence $E$ is a subring of $A$.
Since $E[1/t_1] = E_1[1/t_1] \cap E_2[1/t_1]$, $E_1[1/t_1] = E_1$ and
$E_2[1/t_1] = A_2[1/t_1] = A[1/t_1{t'}] (\supseteq E_1)$, we have 
$E[1/t_1] = E_1= R_1^{[2]}$. Similarly, $E[1/{t'}] = E_2 = R_2^{[2]}$.  
Hence, by Theorem \ref{bcw}, $E = R^{[2]}$ as $(t_1, t')R = R$. 

Note that $E_1[1/t] = E_1[1/t'] = A_1[1/{t'}] = A[1/t]$. Similarly, 
$E_2[1/t] =  A[1/t]$. Hence, $E[1/t] = E_1[1/t] \cap E_2[1/t] = A[1/t]$. 

Let $ I := tA \cap E$. Note that $IE[1/t_1] = tA[1/t_1] \cap  E[1/t_1] = {t'} A_1 \cap  E_1 = (t',f_1)E_1 = (t,f_1)E[1/t_1]$. 
Similarly $I E[1/t'] = (t, f_2)E[1/t']$. Since $(t_1, t')E = E$, we have 
$E/tE \cong E/t_1E \times E/t'E = E[1/t']/(t_1) \times E[1/t_1]/(t')$.   
Let $f \in E$ be a lift of the preimage of the element $(\bar{f_2}, \bar{f_1}) \in E[1/t']/(t_1) \times E[1/t_1]/(t')$.  
Then $I = (t,f)E$. 

Set $A_0 := E[f/t]$. Then $A_0 \subseteq A$, $A_0[1/t_1]= A[1/t_1]$ and $A_0[1/t'] = A[1/t']$. 
Hence $A =A_0 = E[f/t]$.
\end{proof}

\begin{thm}\label{tr1}
Let $k$ be an algebraically closed field of characteristic zero. 
Let $D(\neq 0) \in \lnd_{k[X]}(k[X, X_2, X_3, X_4])$ and $A= $ker$(D)$. 
Suppose that $A$ has only isolated singularities. Then 
there exist distinct elements $\lambda_1, \dots, \lambda_r \in k$, $r \ge 1$, such that
$$
A\cong_{k[X]} k[X,Y, Z, T]/(a(X)Y-F(X, Z, T)) 
$$
where $a(X) = (X-\lambda_1) \cdots (X-\lambda_r)$ and 
$k[Z,T]/(F(\lambda_i, Z,T))$  is a polynomial curve 
for each  $i$, $1\le i \le r$. 
\end{thm}
\begin{proof}
$D$ extends to a locally nilpotent derivation of $k(X)[X_2, X_3, X_4]$
and hence $A \otimes_{k[X]} k(X) = k(X)^{[2]}$ by Theorem \ref{miya}. 
Since ker$(D)$ is finitely generated by \cite[Proposition 4.13]{BD}, 
there exists a monic polynomial $a(X) \in k[X]$ of least possible degree
such that $A[1/a(X)]= k[X, 1/a(X)]^{[2]}$. 
Let $a(X)= (X-\lambda_1) \cdots (X-\lambda_r)$ be a prime factorization of $a$.
Since $a$ has been chosen to be of least possible degree, 
$\lambda_i \neq \lambda_j$ for $i \neq j$. 

Set $R: =k[X]$, $t:= a(X)$ and $t_i:= X-\lambda_i$, $1\le i\le r$;
$B:= k[X, X_2, X_3, X_4]$, $S= R\setminus \bigcup_{i=1}^r t_iR$, 
$R_S:= S^{-1}R$,  $A_S:= S^{-1}A$ and  $B_S:= S^{-1}B=R_S[X_2, X_3, X_4]$. 
Then $R_S$ is a semilocal PID and $D$ induces an $R_S$-linear 
locally nilpotent derivation $D_S$ of $B_S$ with kernel $A_S$. 
Hence, by Lemma \ref{bl3}, there exists a subring $E_S$ of $A_S$ and $f_1 \in E_S$ such that 
$E_S = {R_S}^{[2]}$, $A_S[1/t]=E_S[1/t]$, $t A_S \cap E_S = (t,f_1)E_S$ and $A_S =  E_S[f_1/t]$.

Let $E=E_S \cap A$.  Then, since $A_S[1/t]=E_S[1/t]$, we have, 
$E[1/t] = E_S[1/t] \cap A[1/t] = A[1/t]= R[1/t]^{[2]}$. 
Also $S^{-1} E = S^{-1}E_S \cap S^{-1} A=  E_S = {R_S}^{[2]}$.
Hence, by Theorem \ref{bcw}, $E = R^{[2]}=k[X]^{[2]}$, say $E= k[X][Z,T]$ for some $Z, T\in E$.
 
Since $tA_S \cap A= tA$, we have, $tA \cap E= tA_S \cap E_S \cap A= (t, f_1)E_S \cap E = (t, f_1)S^{-1}E \cap E$.  
Hence $tA \cap E$ is an ideal in $E= k[X,Z, T]$ of height two containing $t$, i.e.,  
$tA \cap E= (t, F)E$ for some $F \in k[X,Z,T]$. Then $(t, F)E_S = (t, f_1)E_S$.

Set $A_0 := k[X,Z,T, F/t] = E[F/t]$. Then $A_0 \subseteq A$, 
$A_0[1/t]= E[1/t] = A[1/t]$ and $S^{-1} A_0 = S^{-1}E[F/t]= E_S[F/t]= E_S[f_1/t]= S^{-1}A$
as $(t, F)E_S = (t, f_1)E_S$. 
Therefore, since $A = A[1/t]\cap S^{-1}A$ and $A_0 = A_0[1/t]\cap S^{-1}A_0$, we have 
$$
A= A_0 \cong_{k[X]}k[X,Y, Z, T]/(a(X)Y-F(X, Z, T)).
$$ 
Since $tA \cap E= (t, F)$, we have $A/tA= k[X,Z,T, Y]/(t, F(X,Z,T))$ and hence 
$A/(X-\lambda_i) = (k[Z,T]/(F(\lambda_i, Z,T)))[Y]$ for each $i$, $1\le i \le r$. 
As $A$ is a factorially closed subring of $B$ (cf. \cite[pg. 22]{F}), 
we have $(X-\lambda_i)B \cap A= (X-\lambda_i)A$ and hence inclusions
$k[Z,T]/(F(\lambda_i, Z,T)) \hookrightarrow A/(X-\lambda_i) \hookrightarrow k[X_2,X_3, X_4]$. Thus, 
$k[Z,T]/(F(\lambda_i, Z,T))$  is a polynomial curve for each  $i$, $1\le i \le r$. 
Hence the result. 
\end{proof}

\begin{rem}
{\em 
Let $A$ be the kernel of a  $k[X_1]$-linear locally nilpotent derivation $D$ of $k[X_1, \dots, X_4]$. 
The first author and Daigle have shown (\cite{BD}) that $A$ is finitely generated. 
Theorem \ref{t5} shows that if $A$ is regular then $A$ is generated by three elements and 
Theorem \ref{tr1} shows that if $A$ has only isolated singularities then $A$ is generated by four elements. 
Recall that Daigle and Freudenburg have exhibited examples (\cite{DF}) to 
show that in general $A$ may require arbitrary number of generators.
}
\end{rem}


\section{Theorem II}\label{g0g1}
In this section, we shall prove Theorem II 
deducing it from a general result (Proposition \ref{propa}). 

Recall that we call an affine one-dimensional domain $E$ over a 
field $k$ an affine rational curve if its field of fractions 
is isomorphic to $k(W)$ and  a polynomial curve if $E\hookrightarrow k[W]$ 
for some $W$ transcendental over $k$. 
We now show that for an affine rational curve $C$ with the normalisation 
$\overline{C}$, $G_0(\overline{C})=G_0(C)$ and $G_1(\overline{C})=G_1(C)$. 

\begin{lem}\label{lemg}
Let $C$ be an affine rational curve over an algebraically closed field $k$ and
$\overline{C}$ denote the normalisation of $C$. Then
\begin{enumerate}
\item [\rm (i)] The group homomorphism $\phi:G_0(\overline{C})\to G_0(C)$ defined by
$\phi([M])=[M]$ is an isomorphism. Hence $G_0(C)$ is a free cyclic group generated by the class $[C]$. 
\item [\rm (ii)] The group homomorphism  $\psi:G_1(\overline{C})\to G_1(C)$ defined by 
$\psi([M,\alpha])=[M, \alpha]$ is an isomorphism. Hence $G_1(C)= \overline{C}^*$. 
\end{enumerate}
\end{lem}
\begin{proof}
(i)  Since $\overline{C}$ is a smooth affine rational curve, 
$\overline{C} \cong k[X,\frac{1}{f(X)}]$ for some $f(X) \in k[X]$.  
Therefore, the rank map $G_0(\overline{C})\to \mathbb{Z}$ is an isomorphism 
as all projective modules over $\overline{C}$ are free and $G_0(\overline{C}) = K_0(\overline{C})$ 
by Remark \ref{ktheory}(ix). 
Thus from the commutative diagram
\[
\xymatrix{
G_0(\overline{C})  \ar@{->}[d]_{\mbox{rank}}^{\cong}\ar@{->}[r]^-{\phi} &  G_0(C)
\ar@{->}[d]^{\mbox{rank}} \\
\mathbb{Z} \ar@{->}[r]^-{\mbox{id}} & \mathbb{Z}, 
} 
\]
we see that the  group homomorphism  $\phi: G_0(\overline{C})\to G_0(C)$ is injective. 
We now show that $\phi$ is surjective.

We first show that $[C/\m]=0$ for every maximal ideal $\m$ of $C$.
Fix a maximal ideal $\m$ of $C$. Then there exists a
maximal ideal $\m'$ of $\overline{C}$ such that $\m=C\cap \m'$. 
By Remark \ref{ktheory}(iii), $[\overline{C}/\m']=0$ in $G_0(\overline{C})$ 
as $\m'$ is a principal ideal in $\overline{C}$. 
Since $k$ is an algebraically closed field, we have $C/\m= \overline{C}/\m'=k$. 
Hence $[C/\m]= \phi([\overline{C}/\m'])=0$ in $G_0(C)$.  

Since dim$(C)=1$ and $[C/\m]=0$ for each maximal ideal $\m$ of $C$, 
it follows from Remark \ref{ktheory}(ii) that 
$G_0(C)$ is a cyclic group generated by the class $[C]$.  
Since $\overline{C}/C$ is a torsion $C$-module, we have $[\overline{C}/C]= 0$ in $G_0(C)$
and hence $\phi([\overline{C}])=[C]$ in $G_0(C)$.   
This proves that $\phi$ is surjective. 

\smallskip

(ii)  We first show that $\psi$ is injective. 
Let $\mathcal{C}$ denote the conductor ideal of $\overline{C}$ over $C$. 
Choose $0\neq b\in \mathcal{C}$. Then $\overline{C}[b^{-1}]=C[b^{-1}]$. 
Since $\overline{C}\cong k[X, 1/f(X)]$ for some $f(X) \in k[X]$, by Remark \ref{ktheory}(ix), 
the map  $G_1(\overline{C}) \to G_1(\overline{C}[b^{-1}])$ is simply the inclusion map 
$\overline{C}^* \hookrightarrow \overline{C}[b^{-1}]^*$.
Thus, from the commutative diagram
\[
 \xymatrix{
G_1(\overline{C})\ar@{^{(}->}[d] \ar@{->}[r]^{\psi} & G_1(C) \ar@{->}[d] \\
G_1(\overline{C}[b^{-1}]) \ar@{->}[r]^{=} & G_1(C[b^{-1}]),
 }
\]
it follows that $\psi$ is injective.

We now prove the surjectivity of $\psi$. 
We first show that $[M,\alpha]=0$ in $G_1(C)$ for any 
finitely generated torsion $C$-module $M$ and any $\alpha \in {\rm Aut}_C (M)$. 
Fix a finitely generated torsion $C$-module $M$ and $\alpha \in {\rm Aut}_C (M)$. 
Since dim $C=1$, 
ann $M= \q_1 \cap \dots \cap \q_r$, where each $\q_i$ is $\m_i$-primary 
ideal for some maximal ideals $\m_i$ of $C$. 
It follows that 
$M = M_1\oplus\ldots\oplus M_r$, where each $M_i$ is a $C/\q_i$-module and 
$\alpha(M_i) = M_i$. Therefore, by Remark \ref{ktheory}(ii), 
$[M,\alpha]=\sum_i [M_i, \alpha|_{M_i}]$. 
Thus, it is enough to consider the case $r=1$, i.e., where 
$\q= \mbox{ann}(M)$ itself is an $\m$-primary ideal for some maximal ideal $\m$ of $C$.
We now show that $[M,\alpha]=0$ in $G_1(C)$ for such an $M$ and $\alpha$. 

Since $\q$ is $\m$-primary, there exists an integer $n\geq 1$ such that 
$\m^n\subseteq \q\subseteq \m$.  Then, 
$$
M\supset \m M \supset \ldots \supset \m^n M=0  \text{~and~} \alpha(\m^i M)=\m^i M, \,  1\leq i\leq n.
$$
Let $N_i= \m^i M/\m^{i+1} M$ and $\alpha_i$ the induced automorphism on $N_i$ for $0\leq i\leq n-1$. 
Then, by Remark \ref{ktheory}(ii), $[M, \alpha]=\sum_i [N_i, \alpha_i]$.
Now each  $N_i$ is a finite-dimensional vector space over the field $C/\m$,
i.e., $N_i = (C/\m)^{r_i}$ for some integer $r_i$. 
Thus, we are reduced to showing that the class $[(C/\m)^{r}, \alpha]=0$
in $G_1(C)$ for any maximal ideal $\m$ of $C$, any integer $r \ge 0$ 
and any automorphism $\alpha$ of $(C/\m)^{r}$. 

Since $\overline{C}$ is integral over $C$, for every
maximal ideal $\m$ of $C$, there exists a maximal ideal $\m'$ of $\overline{C}$ such
that $\m=\m' \cap C$. Since $\overline{C}$ is a PID, 
$\m'$ is principal, say, $\m'= a\overline{C}$.
Since $k$ is an algebraically closed field, $\overline{C}/\m'=k$ and 
hence for any automorphism $\alpha$ of $(\overline{C}/\m')^r$ 
(i.e., a matrix of ${\rm GL}_r(\overline{C}/\m')={\rm GL}_r(k)$),
there exists an automorphism $\tilde{\alpha}$ of $\overline{C}^r$ 
(i.e., a matrix of ${\rm GL}_r(\overline{C})$)
such that we have following commutative diagram:
\[
 \xymatrix{
 0  \ar@{->}[r] & \overline{C}^r \ar@{->}[d]^{\tilde{\alpha}} \ar@{->}[r]^-{(a,\dots,a)} &
\overline{C}^r \ar@{->}[d]^{\tilde{\alpha}} \ar@{->}[r] & (\overline{C}/\m')^r
\ar@{->}[d]^{\alpha} \ar@{->}[r] &  0 \\
 0  \ar@{->}[r] & \overline{C}^r  \ar@{->}[r]^-{(a,\dots,a)} &  \overline{C}^r  \ar@{->}[r]&
(\overline{C}/\m')^r \ar@{->}[r] &  0.
 }
\]
Hence $[(C/\m)^r, \alpha]= \psi([(\overline{C}/\m')^r, \alpha])=\psi([\overline{C}^r, \tilde{\alpha}]-
[\overline{C}^r,\tilde{\alpha}])=0$ in $G_1(C)$.

Now let $M$ be any finitely generated $C$-module and 
$M^{\rm tor}$ be the torsion submodule of $M$. Then 
$0\to M^{\rm tor}\to M\to M/M^{\rm tor}\to 0$ is a short exact
sequence and $M/M^{\rm tor}$ is a torsion-free $C$-module.
For any $C$-automorphism $\alpha$ of $M$, $\alpha(M^{\rm tor})=M^{\rm tor}$. 
Hence, for any $[M, \alpha]\in G_1(C)$, 
$[M,\alpha]=[M^{\rm tor}, \alpha|_{M^{\rm tor}}]+[M/M^{\rm tor},\bar{\alpha}]$, 
where $\bar{\alpha}$ is the automorphism of $M/M^{\rm tor}$ induced by
$\alpha$. 
By above, $[M^{\rm tor},\alpha|_{M^{\rm tor}}]=0$ and hence, for any $[M, \alpha]\in G_1(C)$, 
$[M, \alpha]=[M/M^{\rm tor}, \bar{\alpha}]$. 

Thus, $G_1(C)$ is generated by the classes
$[M, \alpha]$, where each $M$ is a finitely generated torsion-free $C$-module and 
$\alpha$ is an automorphism of $M$. We now show that each $[M, \alpha]$
lies in the image of the map $\psi:G_1(\overline{C})\to G_1(C)$.  

Let 
$K$ be the field of fractions of $C$. We assume that $M$ is torsion-free. Hence, the map 
$M \to M\ot_{C} K$ is injective and so, the map 
$\imath: M\to M\ot_C {\overline{C}}$ is injective.  
Thus, we get the commutative diagram:
\[
 \xymatrix{
0  \ar@{->}[r] & M \ar@{->}[d]^{\alpha} \ar@{->}[r]^-{\imath} &  M\otimes_{C}\overline{C} \ar@{->}[d]^{\alpha\otimes\mbox{id}} 
\ar@{->}[r] & (M\otimes_{C} \overline{C})/{\imath}(M)
\ar@{->}[d]^{\widetilde{\alpha}} \ar@{->}[r] &  0 \\
 0  \ar@{->}[r] & M  \ar@{->}[r]^-{\imath} &  M\otimes_{C}\overline{C} \ar@{->}[r] &
(M\otimes_{C} \overline{C})/{\imath}(M) \ar@{->}[r] &  0.
 }
\]
But $(M\otimes_{C}\overline{C})/{\imath}(M) \cong M\otimes_C (\overline{C}/C)$ is a $C$-torsion
module. Hence $[M\otimes_{C}\overline{C}/{\imath}(M), \widetilde{\alpha}]=0$,
which implies that $[M, \alpha]=\psi([M\otimes_{C}\overline{C}, \alpha\otimes\mbox{id}])$ in $G_1(C)$.
Therefore, $\psi$ is surjective and hence an isomorphism.
\end{proof}
 
We now prove a sufficient condition for an affine curve to be a polynomial curve.
\begin{lem}\label{polycur}
Let $k$ be an algebraically closed field of characteristic zero and $C$ be a one-dimensional affine $k$-domain. 
Suppose that $G_0(C)$ is finitely generated and the canonical map $G_1(k) \to G_1(C)$ is surjective.
Then $C$ is a polynomial curve.  
\end{lem}
\begin{proof}
We first show that  $C$ is an affine rational curve.
Let $\overline{C}$ denote the normalisation of $C$ and 
$\mathcal{C}$ denote the conductor of $\overline{C}$ over $C$. 
Choose $0\neq b\in \mathcal{C}$.
Note that the canonical map $G_0(C)\to G_0(C[1/b])$ is
surjective (cf. Remark \ref{ktheory}(vii)). 
Since $G_0(C)$ is finitely generated, so is $G_0(C[1/b])$. 
Since $C[1/b]=\overline{C}[1/b]$, we have $C[1/b]$ is a regular ring and 
hence $G_0(C[1/b])=K_0(C[1/b])$ by Remark \ref{ktheory}(ix). Therefore, by Lemma \ref{lpic1}, 
$C[1/b]$ is a non-singular affine rational curve. Hence $C$ is an affine rational curve.

We now show that $C$ is a polynomial curve. 
Note that since $C$ is an affine rational curve, the proof of Lemma \ref{lemg} shows
that if $F$ is a finitely generated free $C$-module and 
$\alpha$ is an automorphism of $F$, then 
$\psi ([F \otimes_C  \overline{C}, \alpha \otimes id]) = [F, \alpha]$ in $G_1(C)$. 
Therefore the map $\Gamma : G_1(k) \rightarrow G_1(C)$  which is a composite of
the canonical map $G_1(k) \rightarrow G_1 ( {\overline C})$ and $\psi: G_1({\overline{C}}) \rightarrow G_1(C)$
is the canonical map $G_1(k) \rightarrow G_1(C)$ which is assumed to
be surjective.
Since, by Lemma \ref{lemg}(ii) $ \psi$ is an isomorphism,  the canonical map
$G_1(k) \rightarrow G_1 ({\overline C})$ is surjective.
Since $C$ is an affine rational curve, we have 
$\overline{C} \cong k[X, 1/f(X)]$ for some $f(X) \in k[X]$. 
Since, by Remarks \ref{ktheory}(vi) and (ix), the canonical map 
$G_1(k) \to G_1(\overline{C})$  is simply the inclusion map 
$k^* \hookrightarrow  k[X, 1/f(X)]^*$, it follows that $f(X) \in k^*$. 
Thus, $\overline{C} = k^{[1]}$, i.e., $C$ is a polynomial curve.  
\end{proof}

We now prove the following general result.

\begin{prop}\label{propa}
Let $k$ be an algebraically closed field of characteristic zero and  
$A$ be an affine $k$-domain of dimension $n+1\ge 2$. 
Suppose that there exist $X\in A\setminus k$ and distinct elements
$\lambda_1, \dots, \lambda_r$ in $k$ such that 
\begin{enumerate}
\item [\rm (i)] $A[1/a(X)]= k[X, 1/a(X)]^{[n]}$, where 
$a(X)= (X-\lambda_1) \cdots (X-\lambda_r)$. 
\item [\rm (ii)] For each $i$, $1\le i\le r$, 
$A/(X-\lambda_i) = {C_i}^{[n-1]}$ for some affine domain $C_i$ of dimension one. 
\end{enumerate}
Then the following statements are equivalent:
\begin{enumerate}
\item [\rm (1)] The canonical maps $G_j(k) \to G_j(A)$ are isomorphisms for $j= 0, 1$.
\item [\rm (2)] $C_i$ is a polynomial curve for each $i$, $1\le i \le r$.
\end{enumerate}
\end{prop}
\begin{proof}
Since $A$ is a torsion-free module over the PID $k[X]$, the inclusion $k[X]\hookrightarrow A$ 
is a flat ring homomorphism and hence we have the
commutative diagram (cf. Remark \ref{ktheory}(vii)):
{\tiny
\begin{equation}\label{eqbig}
\xymatrix{
G_2(k[X,1/a(X)]) \ar[r]\ar[d] &G_{1}(\frac{k[X]}{(a(X))})\ar[r]\ar[d] & G_1(k[X]) \ar[r]\ar[d]  & G_1(k[X,1/a(X)]) \ar[r]\ar[d]
&G_{0}(\frac{k[X]}{(a(X))})\ar[r]\ar[d] & G_0(k[X]) \ar@{->>}[r]\ar[d]  & G_0(k[X, 1/a(X)])\ar[d] \\
G_2(A[1/a(X)])  \ar[r]   &G_{1}(\frac{A}{(a(X))})   \ar[r]       & G_1(A)    \ar[r]        & G_1(A[1/a(X)])  \ar[r]    
&G_{0}(\frac{A}{(a(X))})   \ar[r]       & G_0(A)    \ar@{->>}[r]   & G_0(A[1/a(X)]).
}
\end{equation}
}
\normalsize
As $A[1/a(X)]= k[X, 1/a(X)]^{[n]}$, the canonical maps 
$G_j(k[X, 1/a(X)]) \to G_j(A[1/a(X)])$ and $G_j(k) \to G_j(k[X])$ are isomorphisms
for $j=0,1,2$  by Remark \ref{ktheory}(viii). In particular, for $j= 0, 1$, the vertical maps
$G_j(k[X]) \to G_j(A)$ are simply the canonical maps $G_j(k) \to G_j(A)$.
Again, by Remark \ref{ktheory}(ix), the horizontal map $G_1(k[X]) \to G_1(k[X,1/a(X)])$
in the first row is simply the inclusion map $k^* \hookrightarrow k[X,1/a(X)]^*$.
Also from the commutative diagram
\[
\xymatrix{
G_0({k[X]})  \ar@{->}[d]_{\mbox{rank}}^-{\cong}\ar@{->}[r] &  G_0(k[X,1/a(X)])
\ar@{->}[d]_{\mbox{rank}}^{\cong} \\
\mathbb{Z} \ar@{->}[r]^-{\mbox{id}} & \mathbb{Z},
} 
\]
we see that the map $G_0(k[X]) \to G_0(k[X,1/a(X)])$ is an isomorphism. 
Hence, diagram (\ref{eqbig}) yields: 
\tiny
\begin{equation}\label{eqbig1}
\xymatrix{
G_2(k[X,1/a(X)]) \ar[r]\ar[d]_{\cong} &G_{1}(\frac{k[X]}{(a(X))})\ar[r]\ar[d] & G_1(k) \ar@{^{(}->}[r]\ar[d]  & G_1(k[X,1/a(X)]) \ar[r]\ar[d]^{\cong} 
&G_{0}(\frac{k[X]}{(a(X))})\ar[r]\ar[d]& G_0(k) \ar[r]^-{\cong}\ar[d]  & G_0(k[X, 1/a(X)])\ar[d]^{\cong} \\
G_2(A[1/a(X)])  \ar[r] &G_{1}(\frac{A}{(a(X))})   \ar[r]       & G_1(A)    \ar[r]        & G_1(A[1/a(X)])  \ar[r]    
&G_{0}(\frac{A}{(a(X))})   \ar[r]       & G_0(A)    \ar@{->>}[r]        & G_0(A[1/a(X)]).
}
\end{equation}
\normalsize
Set $x_i:= X-\lambda_i$ for $1\le i \le r$.
As $k[X]/(a(X)) = \prod_i k[X]/(x_i)$ and $A/(a(X))= \prod_i A/(x_i)$ by Chinese Remainder Theorem, 
we see that, for $j=0, 1$, $G_j(A/(a(X))) \cong \oplus_{i=1}^{r}G_j(A/(x_i))$, and 
the vertical maps $G_j(k[X]/(a(X))) \to G_j(A/(a(X)))$ are simply the product of the canonical maps 
$G_j(k) (= G_j(k[X]/(x_i))) \to G_j(A/(x_i))$.   
Since $A/(x_i) = {C_i}^{[n-1]}$, by Remark \ref{ktheory}(viii), the canonical map 
$G_j(C_i)\hookrightarrow G_j(A/(x_i))$  are  isomorphisms for $j=0, 1$.
Hence, each vertical map $G_j(k[X]/(a(X))) \to G_j(A/(a(X)))$, for $j= 0, 1$,
is the product of the canonical maps $G_j(k) \to G_j(C_i)$.   

We now prove the equivalence of (1) and (2).

\noindent
(1) $\implies$ (2): 
Since the vertical maps $G_j(k) \to G_j(A)$ are isomorphisms for $j=0, 1$, it 
follows from diagram (\ref{eqbig1}), that 
the vertical maps $G_j(k[X]/(a(X))) \to G_j(A/(a(X)))$ are surjective for $j= 0, 1$. Hence, 
the canonical maps $G_j(k) \to G_j(C_i)$ are surjective for $j= 0, 1$ and $1\le i\le r$. 
Since $G_0(k)$ is isomorphic to $\BZ$, we have $G_0(C_i)$ is finitely generated for $1\le i\le r$.
The result now follows by Lemma \ref{polycur}.

\smallskip

\noindent
(2) $\implies$ (1): 
Fix $i$,   $1\le i \le r$.
Let $\overline{C_i}$ denote the normalisation of $C_i$. 
Since $C_i$ is a polynomial curve, $\overline{C_i}=k[W]$ for some $W$ by Theorem \ref{thE} and 
hence the canonical maps $G_j(k) \to G_j(\overline{C_i})$ are isomorphisms, for $j= 0, 1$,
by Remark \ref{ktheory}(viii).  Moreover, by Lemma \ref{lemg}, 
the group homomorphisms $\phi:G_0(\overline{C_i})\to G_0(C_i)$ and 
$\psi:G_1(\overline{C_i})\to G_1(C_i)$ are isomorphisms. 
Thus the composite map
$$
 G_j(k) \rightarrow G_j({\overline{C_i}}) \rightarrow G_j(C_i) \rightarrow G_j(A/(X- \lambda_i))
$$
is an isomorphism for $j =0,1$. Note that these maps are the
canonical maps induced by the inclusion $k[X] \rightarrow A$. Hence,
the vertical maps, $G_{j}(\frac{k[X]}{(a(X))}) \to G_{j}(\frac{A}{(a(X))})$ are isomorphisms for $j=0, 1$.
From diagram (\ref{eqbig1}), it now follows that
the maps $G_j(k) \to  G_j(A)$ are isomorphisms for $j=0, 1$. Hence the result. 
\end{proof}

The above result applies to the family of threefolds $x^my=F(x,z,t)$ considered in \cite{G1};
in particular, to the Russell-Koras cubic $x^2y+x+z^2+t^3=0$.

\begin{cor}\label{polyk}
Let $k$ be an algebraically closed field of characteristic zero  and \\
$A= k[X,Y,Z,T]/(X^mY-F(X,Z,T))$, where $m \ge 1$
and $F(0, Z,T)$ is an irreducible polynomial in $k[Z,T]$. 
Then the canonical maps $G_j(k) \to G_j(A)$ are isomorphisms for $j= 0, 1$ 
if and only if  $C:= k[Z,T]/(F(0, Z,T))$ is a polynomial curve.
In particular, if $A$ is the Russell-Koras cubic $\BC[X,Y,Z,T]/(X^2Y+X+Z^2+T^3)$, then 
$K_0(A)= G_0(A)= \BZ$ and $K_1(A)= G_1(A)= k^*$. 
\end{cor}
\begin{proof}
The equivalence follows from Proposition \ref{propa} by taking $a(X)= X$. 
For the Russell-Koras cubic $A$ (a regular ring), the equality of $G_i(A)$ and $K_i(A)$ for $i=0,1$
follows from Remark \ref{ktheory}(ix). 
\end{proof}

As an immediate application of Proposition \ref{propa}, we deduce the following result 
which was proved in \cite[Theorem 3.11]{G1} for $m \ge 2$.

\begin{cor}\label{comp}
Let $k$ be an algebraically closed field of characteristic zero and \\
$A= k[X,Y,Z,T]/(X^mY-F(X,Z,T))$, where $m \ge 1$. Then the following statements are equivalent:
\begin{enumerate}
\item [\rm (i)] $A = k^{[3]}$. 
\item [\rm (ii)] $k[Z,T]= k[F(0, Z,T)]^{[1]}$. 
\item [\rm (iii)] $k[Z,T]/(F(0, Z,T))= k^{[1]}$. 
\end{enumerate}
\end{cor}
\begin{proof}
The equivalence of (ii) and (iii) is the Epimorphism Theorem  of Abhyankar-Moh-Suzuki (cf. \cite{AM} and \cite{Suz}).

\noindent
(iii) $\implies$ (i). Let $x$ denote the image of $X$ in $A$. 
Since $k[Z,T]/(F(0,Z,T)) = k^{[1]}$, we have 
$A/xA= k^{[2]}$. 
Set $R:= k[x]$. Let $P$ be a maximal ideal of $R$ and set $S:= R\setminus P$. 
Let $R_P = S^{-1}R$ and $A_P = S^{-1}A$. If $x \notin P$, then clearly $A_P = {R_P}^{[2]}$.
Now suppose that $x \in P$, i.e., $P= xR$. 
Since $\BQ \subseteq A_P$ and $A/xA= k^{[2]}$, by Theorem \ref{Sa}, we have $A_P = {R_P}^{[2]}$. 
Therefore, since $R$ is a PID, by Theorem \ref{bcw}, $A= R^{[2]}= k^{[3]}$. 

\noindent
(i) $\implies$ (iii).
Set $f(Z,T): = F(0, Z, T)$ and $C:= k[Z,T]/(f)$.  We show that $C= k^{[1]}$. 
If $m \ge 2$, then we are through by \cite[Theorem 3.11]{G1}. Now suppose that $m=1$. 
Since $A$ is a regular UFD, it follows that $f$ is an irreducible polynomial in $k[Z,T]$ 
(cf. \cite[Lemma 3.1]{G1}) and $C$ is a regular domain. 
By Remark \ref{ktheory}(viii), the canonical maps $G_j(k) \to G_j(A)$ are isomorphisms for $j=0, 1$. 
Hence, by Proposition \ref{propa}, $C$ is a polynomial curve and so, by Theorem \ref{thE}, 
we have $C= k^{[1]}$. 
\end{proof}

We now prove Theorem II.

\begin{thm}\label{th7}
Let $k$ be an algebraically closed field  of characteristic zero.  Let $D$ be a $k[X]$-linear 
locally nilpotent derivation of $k[X, X_2, X_3, X_4]$ and $A= $ker$(D)$. 
Then the canonical maps $G_j(k) \to G_j(A)$ are isomorphisms for $j= 0, 1$.
In particular, $G_0(A)= \BZ$ and $G_1(A)= k^*$. 
\end{thm}
\begin{proof}
We may assume that $D \neq 0$. $D$ extends to a locally nilpotent derivation of\\ $k(X)[X_2, X_3, X_4]$
and hence $A \otimes_{k[X]} k(X) = k(X)^{[2]}$ by Theorem \ref{miya}. 
Since ker$(D)$ is finitely generated by \cite[Proposition 4.13]{BD}, 
there exists a monic polynomial $a(X) \in k[X]$ of least possible degree such that $A[1/a(X)]= k[X, 1/a(X)]^{[2]}$. 
Let $a(X)= (X-\lambda_1) \cdots (X-\lambda_r)$ be a prime factorization of $a$.
By Lemma \ref{BD}, for each $i$, $1\le i \le r$, 
$A/(X-\lambda_i) = {C_i}^{[1]} \hookrightarrow k[X_2, X_3, X_4]$ for some subring $C_i$ of $A/(X-\lambda_i)$. 
Hence $C_i$ is a polynomial curve. 
The result now follows from Proposition \ref{propa}. 
\end{proof}

\section{Theorem III}\label{Main}
In this section we shall prove Theorem III (Theorem \ref{MT}) which will provide an
infinite family of counter-examples to Miyanishi's question. 

Let $F(Z, T) \in k[Z, T]$ be an irreducible polynomial, $C=k[Z, T]/(F(Z, T))$ and 
$A=k[U, V, Z, T]/(U^mV-F(Z, T))$, where $ m\geq 1$.  
We write $A=k[u, v, z, t]$, where $u, v, z, t$ are the 
images of $U, V, Z, T$ respectively. 
We prove below two lemmas relating the finite generation of the two groups 
$K_0(A)$ and $K_0(C)$ with that of $K_0(\mathcal{P}_{1}(A,u))$.

\begin{lem}\label{l2}
If $K_0(\mathcal{P}_{1}(A,u))$ is not finitely generated then $K_0(A)$ is not finitely generated.
\end{lem}
\begin{proof}
By Theorem \ref{locseq}, we have the following exact sequence
$$
K_1(A)\rightarrow K_1(A[1/u]) \rightarrow K_0(\mathcal{P}_{1}(A,u))\rightarrow {K_0}(A)\rightarrow {K_0}(A[1/u]). 
$$
The above sequence reduces to 
\[
 \xymatrix{ 0\ar[r] & H \ar[r] & K_0(\mathcal{P}_{1}(A,u))\ar[r] & {K_0}(A)\ar[r] & {K_0}(A[1/u]),  }
\] 
where $H:=$ coker$(K_1(A)\rightarrow K_1(A[1/u]))$. 
To prove the result, it is enough to show that $H$ is a finitely generated group.
Now $A[1/u]= k[u, 1/u, z, t]$ and hence $K_1(A[1/u]) = K_1(k[u, 1/u]) = k[u,1/u]^*$ (cf. Remark \ref{ktheory}(xi) and (viii)).  
Note that the composite map $K_1(k) \to K_1(A) \to K_1(A[1/u])$ is simply the inclusion map
$k^* \hookrightarrow  k[u,1/u]^*$.  
Hence we have the following exact sequence:
\begin{equation}
K_1(A)/k^* \to k[u, 1/u]^*/k^* \to H \to 0. 
\end{equation}
Thus $H$ is the quotient of the group $k[u, 1/u]^*/k^* \cong \BZ$ and hence a finitely generated group.
\end{proof}

\begin{lem}\label{l3}
If $K_0(C)$ is not finitely generated then $K_0(\mathcal{P}_{1}(A,u))$ is also not finitely generated.
\end{lem}
\begin{proof}
Set $E:=A/vA$. Let $\phi: A \rightarrow E$  denote  the canonical
surjective  map sending $v$ to $0$. We write $u$ for $\phi(u)$ in $E$. 
Then $E= C[u]$ and $E/uE= C$.
Then, by Lemma \ref{ll1}, the map 
$K_0(\widetilde{\phi}): K_0(\mathcal{P}_{1}(A,u))\to K_0(\mathcal{P}_{1}(E,u))$ sending 
$[M]$ to $[M\ot_A E]$, is well defined. 
Let $\overline{\phi}: A/uA \to E/uE$ be  the map induced by $\phi$ going modulo $u$.
Then $\overline{\phi}$ induces  a canonical map 
$K_0({\overline{\phi}}): K_0(A/uA)\to K_0(E/uE)$. 
Set $\phi_1:= K_0({\overline{\phi}})$ and $\phi_2:= K_0({\widetilde{\phi}})$ and 
let $\sigma_1$ and $\sigma_2$ be the natural maps defined in equation (\ref{eqka}) of Section \ref{prelim}. 
Let $\bar{v}$ denote the image of $v$ in $A/uA$. Hence, as $A/uA=C[\bar{v}]$ and $E/uE= C$, 
we have the following commutative diagram:
\[
\xymatrix{
K_0(C[\bar{v}])= K_0(A/uA) \ar@{->}[r]^-{\sigma_1}  \ar@{->}[d]^{\phi_1} & K_0(\mathcal{P}_{1}(A,u)) 
     \ar@{->}[d]^{\phi_2} 
\\
K_0(C)= K_0(E/uE)  \ar@{->}[r]^-{\sigma_2}  &K_0(\mathcal{P}_{1}(E,u)).
}
\]
Since the map $A/uA\longrightarrow E/uE$ maps $\bar{v}$ to $0$, 
the map $\phi_1$ is surjective. 
Hence Im$(\sigma_2)=$~Im$(\phi_2\sigma_1)$ and so 
Im$(\sigma_2)\subseteq $ Im$(\phi_2)$.   
Since $E=C[u]$, by Lemma \ref{l1}, $\sigma_2$ is injective.
Thus, $K_0(C)$ is isomorphic to a subgroup of $\phi_2(K_0(\mathcal{P}_{1}(A,u)))$. 
Therefore, if $K_0(C)$ is not finitely generated  
then  $K_0(\mathcal{P}_{1}(A,u))$ is also not finitely generated.
\end{proof} 

We now establish that the group $K_0(A)$ is finitely generated if and only if $K_0(C)$ is so. 
The implication (c) $\implies$ (a) of the following proposition will be used in the proof of Theorem III.

\begin{prop}\label{pm}
Let $k$ be an algebraically closed field of characteristic zero. 
Let $F(Z, T)$ be an irreducible polynomial of $k[Z, T]$, 
$C= k[Z, T]/(F)$ and 
$A= k[U, V, Z, T]/(U^mV-F(Z,T))$ for some integer $m \ge 1$. 
Then the following statements are equivalent:
\begin{enumerate}
\item [\rm(a)] $C$ is a non-singular affine rational curve. 
\item [\rm(b)] {\it ${K_0}(A) =\mathbb Z$.}
\item [\rm(c)] {\it ${K_0}(A)$ is finitely generated.}
\item [\rm(d)] {\it ${K_0}(C)$ is finitely generated.}
\end{enumerate}
\end{prop}
 \begin{proof}
(a) $\implies$ (b): Since $C$ is a regular ring, it follows that $A$ is a regular ring and $A/uA$ is a regular ring. 
Hence, by Remark \ref{ktheory}(vii) and (ix), we have the following exact sequence:
\begin{equation}\label{mm}
 K_1(A)\rightarrow K_1(A[1/u]) \rightarrow K_0(A/uA)\stackrel{\delta}{\rightarrow}
 {K_0}(A)\rightarrow {K_0}(A[1/u]) \rightarrow 0. 
\end{equation}
Now $C$ is a non-singular affine rational curve and hence 
$C$ is of the form $k[X, \frac{1}{f(X)}]$ for some $f(X) \in k[X]$. 
Thus $K_0(C)= \mathbb Z$ and since $A/uA \cong C^{[1]}$, 
we have $K_0(A/uA) = \mathbb Z$ (cf. Remark \ref{ktheory}(viii)). Hence, $K_0(A/uA)$ is generated by $[A/uA]$,
the class of the rank-one free module $A/uA$.
Since $u$ is a nonzerodivisor in $A$, the following sequence is exact:
$$
0 \rightarrow A\stackrel{u}{\rightarrow} A\rightarrow A/uA\rightarrow 0.
$$ 
Therefore, by Remark \ref{r3}, $\delta([A/uA]) = [A]-[A]=0$, i.e., $\delta$ is the zero map. Hence,
by (\ref{mm}), ${K_0}(A) \cong {K_0}(A[1/u])$. 
Now $A[1/u]= k[u, 1/u, z, t]$ and hence $K_0(A[1/u]) =\mathbb Z$. 
Therefore, $K_0(A) = \mathbb Z$. 

\smallskip

(b) $\implies$ (c): Trivial. 

\smallskip

(c) $\implies$ (d):
Follows from Lemmas \ref{l2} and \ref{l3}.

\smallskip

(d) $\implies$ (a): 
Follows from Lemma \ref{lpic1}. 
\end{proof}

We now prove Theorem III. Recall that for an $R$-module $M$, an element $x\in M \setminus \{0\}$ is 
said to be an {\it unimodular element} if $M= N \oplus Rx$  for some submodule $N$ of $M$, i.e., if
there exists $f \in {\rm Hom}_R\,(M, R)$ satisfying $f(x)=1$.

\begin{thm}\label{MT}
Let $k$ be an algebraically closed field of characteristic zero.
Let $F(Z, T)$ be an irreducible polynomial of $k[Z, T]$, 
$C= k[Z, T]/(F)$, $B=k[X_1, X_2, X_3, X_4]$ and 
$$
A= k[X_1, Y, Z, T]/({X_1}^mY-F(Z, T)) \text{~~where~~} m \ge 1.
$$
Then 
\begin{enumerate}
\item [\rm (i)] Every projective $A$-module of rank $\geq 3$ has a unimodular element.
\item [\rm (ii)] If $C$ is not a regular ring, then $K_0(A)$ is not finitely generated; in fact, 
there exists an infinite family of pairwise non-isomorphic projective $A$-modules of rank two
which are not even stably isomorphic.  
\item [\rm (iii)] The following statements are equivalent:\\
(a)  There exists a $D \in \lnd_{k[X_1]}(B)$ such that $A \cong$ ker$(D)$ as $k$-algebras.\\
(b) $C$ is a polynomial curve.
\end{enumerate}
\end{thm}
\begin{proof}
(i)  Since $k[X_1, Z, T] \subset A \subset k[X_1, {X_1}^{-1}, Z,T]$,
from \cite[Theorem 4.2]{BR}  it follows that every projective $A$-module of
rank $\geq 3$ has a unimodular element. 

\smallskip

(ii) Since $C$ is a non-regular ring, 
by Proposition \ref{pm}, $K_0(A)$ is not finitely generated. 
Since $A[1/X_1] (\cong k[X_1, 1/X_1]^{[2]})$ is a UFD and 
$A/X_1A$ is an integral domain, by Nagata's criterion 
$A$ is a UFD (cf. \cite[Theorem 20.2]{Mat}) and hence all projective $A$-modules of rank one are free. 
Therefore, as $K_0(A)$ is not finitely generated, by (i) there exists
an infinite family of pairwise non-isomorphic projective $A$-modules of rank two
which are not even stably isomorphic. 

\smallskip

(iii) (b) $\implies$ (a): Let $C \hookrightarrow k[W]$ for some $W$ transcendental over $k$. 
Let $z$ and $t$ denote respectively the images of $Z$ and $T$ in $C$. Then 
$z = \alpha(W)$ and $t =\beta(W)$ for some $\alpha(W)$, $\beta(W) \in k[W]$ and    
$F(\alpha(W), \beta(W)) = 0$.
In particular,  $F(\alpha(X_2), \beta(X_2))=0$.  
Therefore, we have a  $k[X_1]$-algebra homomorphism $\phi:~k[X_1, Y, Z, T]\rightarrow B$
defined by 
\begin{eqnarray*}
\phi(Z)&=&\alpha(X_2)-X_{1}^{m}X_3,\\
\phi(T)&=&\beta(X_2)-X_{1}^{m}X_4,\\
\phi(Y)&=& F(\phi(Z), \phi(T))/X_1^m.
\end{eqnarray*}
Now ${X_1}^mY-F(Z, T)\in $ ker $\phi$ is an irreducible polynomial of $k[X_1, Y, Z, T]$. 
Since tr.deg$_k({\rm Im}(\phi))=3$,
ker $\phi$ is a height one prime ideal. 
Hence ker $\phi= ({X_1}^mY-F(Z, T))$. 
Therefore, $\phi$ induces an injective ring homomorphism 
$\tilde{\phi}:A\rightarrow B$. We now identify $A$ with the subring $\tilde{\phi}(A)$ of $B$
and let $y=\phi(Y), z=\phi(Z), t=\phi(T)$. 
Consider the triangular derivation $D$ of $B$ defined by: 
$$
D=X_{1}^{m} \partial_{X_2}+\alpha'(X_2)\partial_{X_3}+\beta'(X_2)\partial_{X_4} \in \lnd_k(B).
$$
Let $A_0=$ ker$(D)$.  We show that $A=A_0$.
Clearly, $A\subseteq A_0$. We observe that $B[{X_1}^{-1}]=A[{X_1}^{-1}][X_2]$
and hence $A[{X_1}^{-1}]$ is algebraically closed in $B[{X_1}^{-1}]$. 
It then follows that $A[{X_1}^{-1}]=A_0[{X_1}^{-1}]$. Let $P=X_1B \cap A$. 
Then $A/P= k[\alpha(X_2), \beta(X_2), \bar{y}]$, 
where $\bar{y}$ is the image of $y$ in $B/X_1B= k[X_2, X_3, X_4]$.
Note that 
$$
\bar{y}= -X_3 F_Z(\alpha(X_2), \beta(X_2)) -X_4 F_T (\alpha(X_2), \beta(X_2)),
$$
and that $F$ being a non-constant polynomial, either $F_Z$ or $F_T$ is a non-zero polynomial in $k[Z, T]$.  
Moreover, since $F(Z, T)\in k[Z, T]$ generates the kernel of the $k$-algebra homomorphism 
$k[Z, T]\rightarrow k[\alpha(X_2), \beta(X_2)]$, it follows that 
either $F_Z(\alpha(X_2), \beta(X_2)) \neq 0$ or $F_T(\alpha(X_2), \beta(X_2)) \neq 0$.
Therefore,
$\bar{y}$ is transcendental over $k[\alpha(X_2), \beta(X_2)]$.
Thus tr.deg$_k(A/P)=2$ and hence the  height of $P$ is $1$. Thus $P= X_1 A$.
Therefore, the canonical map $A/X_1A \to B/X_1B$ is injective 
and hence the canonical map $A/X_1A \to A_0/X_1A_0$ is injective.
Hence, by Lemma \ref{lemlast}, $A=A_0=$ ker$(D)$. 

\noindent
(a) $\implies$ (b):
Let $x_1$ denote the image of $X_1$ in $A$.  Since $A[1/x_1]=k[x_1, 1/x_1]^{[2]}$, 
$A/x_1A = C^{[1]}$ and $C$ is an integral domain, the result 
follows from Theorem \ref{th7} and Proposition \ref{propa}.
\end{proof}

\begin{rem}\label{r11}
{\em  (a) If $\deg_W(\alpha(W))$ and $\deg_W(\beta(W))$ are at least two and \\
${\rm gcd}(\mbox{deg}(\alpha(W)), \mbox{deg}(\beta(W)))=1$, 
then $C=k[\alpha(W), \beta(W)]\hookrightarrow k[W]$ is a non-regular ring birational to $k[W]$. 
Thus there exist many rings $C$ which satisfies (ii) and (iii)(b) of Theorem \ref{MT}.
If gcd$(\alpha'(W), \beta'(W))=1$ then the derivation 
$D$ in Theorem \ref{MT} (iii) is fixed point free.

(b) Consider a polynomial $F(Z, T) \in k[Z, T]$ such that 
the ring $k[Z, T]/(F(Z, T)) \cong k[\alpha(W), \beta(W)](\hookrightarrow k[W])$ 
has exactly one singularity and satisfies gcd$(\alpha'(W), \beta'(W))=~1$.
Then, for such a polynomial $F(Z, T)$, 
the ring $A=k[X_1, Y, Z, T]/(X_1Y-F(Z, T))$ has exactly one singularity. 
By Theorem \ref{MT}(iii) and (a), there exists a fixed point free  $k[X_1]$-linear 
locally nilpotent derivation $D$ of $k[X_1]^{[3]}$ such 
that the ring $A$ is isomorphic to ker$(D)$ and $K_0(A)$ is not finitely generated. 
An explicit example of such an $F$ and $D$ is given below (Example \ref{ex}).  
However, if $D$ is a  $k[X_1]$-linear locally nilpotent derivation of $k[X_1, X_2, X_3, X_4]$ 
such that the kernel of $D$ is a regular ring,
then all projective modules over ker$(D)$ are free (cf. Theorem \ref{t5}). 

(c) We can extend the derivation in Theorem \ref{MT}(iii) to $k[X_1,\ldots, X_n]$ 
for any $n\geq 5$ by setting $DX_i=0$ for $i\geq 5$. Then 
ker$(D)=A[X_5,\ldots, X_n]$. Since $K_0(A)$ is not finitely generated, 
$K_0(A[X_5,\ldots, X_n])$ is also not finitely generated. Hence for any 
$n\geq 4$, there exists a locally nilpotent derivation on $k^{[n]}$ 
such that projective modules over it's kernel need not be free.

(d) Let $k$ be a field of characteristic zero not necessarily algebraically closed
and $F(Z,T) = Z^2+T^3 \in k[Z,T]$. Then the Picard group of $C:= k[Z,T]/(F)$ is 
the group $(k,+)$ which is not finitely generated. 
It follows from Lemmas \ref{l2} and \ref{l3} that $K_0(A)$ is not finitely generated
where $A= k[U,V, Z,T]/(U^mV-F(Z,T))$. Since 
$C$ is a polynomial curve, it follows from the proof of Theorem \ref{MT}(iii) 
that $A$ is isomorphic to the kernel of a locally nilpotent derivation of $k[X_1, X_2, X_3, X_4]$. 
}
\end{rem}

We now give an explicit fixed point free  $k[X_1]$-linear locally nilpotent derivation $D$ 
of $k[X_1, X_2, X_3,X_4]$ such that the Grothendieck group of ker$(D)$ is not finitely generated.  

\begin{ex}\label{ex}
{\em  Let $n >1$ be an integer, 
$\alpha(W)=W^n$, $\beta (W) = W(W^n +1)$ and  $F(Z,T) =Z (Z+1)^n -T^n$.
Then $k[Z, T]/(F(Z,T)) \cong k[\alpha(W), \beta(W)]$. 
Let $A=k[X_1, Y, Z, T]/({X_1}Y-F(Z, T))$. 
Then $D=X_{1} \partial_{X_2}+(nX_{2}^{n-1})\partial_{X_3}+((n+1)X_{2}^{n}+1)\partial_{X_4}\in$ 
$\lnd_k(k[X_1,\cdots,X_4])$ is a fixed point free derivation with $A\cong \mbox{ker}(D)$
which has only an isolated singularity.
By Proposition \ref{pm}, $K_0(A)$ is not finitely generated. 
In particular, there exist infinitely many non-free (in fact, non-isomorphic) projective modules 
over the kernel of the locally nilpotent derivation $D$. }
\end{ex}

\noindent
{\bf Acknowledgements}
The authors thank Amartya K. Dutta for carefully going through the draft and suggesting 
improvements. The second author also acknowledges Department of Science and Technology for their INSPIRE Research Grant. 

{\small
{}}
\end{document}